\documentclass[12pt]{article}
\usepackage{longtable}
\usepackage[top=1in,bottom=1in,left=1in,right=1in]{geometry}
\usepackage{hyperref}
\hypersetup{
	hypertex=true,
	colorlinks=true,
	linkcolor=blue,
	anchorcolor=blue,
	citecolor=blue
}
\usepackage{xcolor}
\usepackage{caption} % 在导言区添加

\usepackage{caption}
\usepackage{longtable}
\usepackage{array}

\usepackage{indentfirst}
\usepackage{latexsym}
\usepackage{algorithm}
\usepackage{algorithmic}
\usepackage{setspace}
\usepackage{longtable, booktabs}
\usepackage{tabu}
\usepackage{amsthm}
\usepackage{amsmath}
\usepackage{amssymb}
\usepackage{supertabular}
\usepackage{mathrsfs}

\usepackage{rotating}
\usepackage{multicol}
\usepackage{multirow}
\usepackage{makeidx} % additional command see
\usepackage{epsfig}
\usepackage{fancyhdr}       % fuer Deckblatt
\usepackage{hyperref}
\allowdisplaybreaks[4]

\newtheorem{theorem}{Theorem}[section]
\newtheorem{lemma}[theorem]{Lemma}

\newtheorem{corollary}[theorem]{Corollary}
\newtheorem{remark}[theorem]{Remark}

\def\lam{\lambda}

\def\Soc{{\rm Soc}}

\def\G1{G^\mathcal{C}}

\def\D{\mathcal{D}}

\DeclareMathOperator{\Aut}{Aut} 
 \DeclareMathOperator{\Out}{Out}


\begin{document}
	\title{Block-transitive $t$-($k^2,k,\lambda$) designs with $PSL(n,q)$ as socle}
	\author{Guoqiang Xiong$^1$,Haiyan Guan$^{1,2}$\footnote{This work is supported by the National Natural Science Foundation
	of China (Grant No.12271173 ).} \\
		{\small\it  1. College of Mathematics and Physics, China Three Gorges University, }\\
		{\small\it  Yichang, Hubei, 443002, P. R. China}\\
		{\small\it 2. Three Gorges Mathematical Research Center, China Three Gorges University,}\\
		{\small\it  Yichang, Hubei, 443002, P. R. China}\\
		\date{}
	}
	
	\maketitle
	\date

\begin{abstract}
Let $\mathcal{D}=(\mathcal{P},\mathcal{B})$ be a non-trivial block-transitive $t$-$(k^2,k,\lambda)$ design with $G\leq \Aut(\mathcal{D})$ and $X\unlhd G\leq \Aut(X)$, where $X=PSL(n,q)(n\geq3).$ We prove that $t=2$ and the parameters $(n,q,v,k)$ is $(3,3,144,12),(4,7,400,20)$ or $(5,3,121,11).$ Moreover, $\mathcal{D}$ is a $2$-$(144,12,\lambda)$ design with $\lambda\in\{3,6,12\}$ if $\lambda\mid k$. 
	
	
\end{abstract}
	\section{Introduction}
	 A $t$-$(v,k,\lambda)$ design $\mathcal{D}$ is a pair $(\mathcal{P},\mathcal{B})$, where $\mathcal{P}$ is a set of  $v$ points and $\mathcal{B}$ is a collection of $k$-subsets of $\mathcal{P}$(called blocks), such that any  $t$ distinct points are contained in exactly $\lambda$ blocks. We say $\mathcal{D}$ is non-trivial if $t<k<v$.
	 All $t$-$(v,k,\lambda)$ designs in this paper are assumed to be non-trivial. An automorphism of $\mathcal{D}$ is a permutation of $\mathcal{P}$ that preserves $\mathcal{B}$. The full automorphism group of $\mathcal{D}$ comprises all such automorphisms, forming a group under the operation of permutation composition, denoted by $\Aut(\mathcal{D})$  and any subgroups of  $\Aut(\mathcal{D})$ is called an automorphism group of $\mathcal{D}$. 
	 A flag of $\mathcal{D}$ is a pair ($\alpha$,$B$) that $B$ contains $\alpha$ with $\alpha\in \mathcal{P}$, $B\in\mathcal{B}.$
	 A group $G$ is described as point-primitive (point-transitive, block-transitive, or flag-transitive, respectively) according to whether it acts primitively on points (transitively on points, transitively on blocks, or transitively on flags, respectively). By Block's result in \cite[Corollary 2.2]{block1967orbits}, any block-transitive group $G$ is necessarily point-transitive. If $X\unlhd G\leq \Aut(X)$ for some nonabelian simple group $X$, then $G$ is said to be almost simple with socle $X$.
	 %If a subgroup $G$ of $\Aut(\mathcal{D})$  acts transitively on the point set $\mathcal{P}$, then the design $\mathcal{D}$ is called point-transitive. Similarly, if G acts transitively on the block set $\mathcal{B}$, then $\mathcal{D}$ is called block-transitive. 
	
	 %Notably, according to the result of Block(\cite{block1967orbits}), if $\mathcal{D}$ is block-transitive, it is also point-transitive. For any automorphism group $G$, $G$ is described as point-primitive if acts primitively on points $\mathcal{P}$. Conversely, it is termed point-imprimitive.

	  The research on $t$-$(k^2,k,\lambda)$ designs has been started by Montinaro and Francot in their work(\cite{MR4561672,MR4516389,MR4481043}). More specifically, they demonstrated that for a  2-$(k^2,k,\lambda)$ design  with $\lambda$ $\mid$ $k$, any flag-transitive automorphism group $G$ is either an affine group or an almost simple group. The almost simple type was classified in \cite{MR4516389,MR4481043}, and the  affine type in \cite{MR4561672}. On the other hand, compared to flag-transitivity, block-transitivity is a weaker condition. The characterization of $t$-$(k^2,k,\lambda)$ designs admitting block-transitive automorphism groups presents greater challenges compared to the flag-transitive case. 
	  Guan and Zhou(\cite{MR4830077}) established that if $G$ is a block-transitive automorphism group of a $t$-$(k^2,k,\lambda)$ design, then $G$ must act primitively on the points $\mathcal{P}$. In addition, $G$ is either an affine type or an almost simple type. Such designs admitting an almost simple automorphism group with socle $X$ being a finite simple exceptional group of Lie type, a sporadic simple group or an alternating group have been  studied in \cite{Chen,MR4830077}. Alternatively, Xiong and Guan(\cite{Xiong}) have established the classification for the specific case where $X = PSL(2, q)$. Here, we continue this work, and investigate block-transitive $t$-$(k^2,k,\lambda)$ designs where	$X=PSL(n,q)$ with $n\geq3$. It is well-known that if $G$ is a transitive permutation group acting on $\mathcal{P}$, then for any non-empty subset $B \subseteq \mathcal{P}$, the pair $(\mathcal{P}, B^G)$ constructs a block-transitive $1$-design. Thus, we always assume that $t \geq 2$ in our work. The following is our main result:
	 \begin{theorem}\label{th1.2}
	 	\textup{Let $\mathcal{D} = (\mathcal{P},\mathcal{B})$ be a non-trivial $t$-$(k^2,k,\lambda)$ design admitting a block-transitive automorphism $G$ and $X\unlhd G\leq \Aut(X)$ with $X = PSL(n,q)$ and $n\geq3$. Suppose that  $\alpha\in\mathcal{P}$, $G_\alpha$ is  the point-stabilizer of $G$. Then $t=2$ and one of the  following holds:}
	 		\begin{enumerate}
	 		\item[{\rm(1)}]\textup{$X\cap G_\alpha\cong \mkern-1mu \raisebox{0.5ex}{$\hat{\ }$} \mkern-3mu (q^2+q+1):3$, $(n,q,v,k)=(3,3,144,12)$;}
	 		\item[{\rm(2)}]\textup{$X\cap G_\alpha\cong \mkern-1mu \raisebox{0.5ex}{$\hat{\ }$} \mkern-3mu[q^3]:GL(3,q)$, $(n,q,v,k)=(4,7,400,20)$;}
	 		\item[{\rm(3)}]\textup{$X\cap G_\alpha\cong \mkern-1mu \raisebox{0.5ex}{$\hat{\ }$} \mkern-3mu[q^4]:GL(4,q)$, $(n,q,v,k)=(5,3,121,11)$.}
	 			\end{enumerate}
	 \end{theorem}
\begin{remark}
{\rm Throughout this article, we sometimes precede the structure of a subgroup of a projective group with the symbol $\mkern-1mu \raisebox{0.5ex}{$\hat{\ }$} \mkern-3mu$ , which indicates that we are referring to the structure of its pre-image in the corresponding linear group. }
\end{remark} 
From Theorem \ref{th1.2} and the work of Xiong and Guan \cite[Theorem 1.1]{Xiong}, we directly derive Corollary \ref{13}.
\begin{corollary}\label{13}
		 	\textup{Let $\mathcal{D} = (\mathcal{P},\mathcal{B})$ be a non-trivial $t$-$(k^2,k,\lambda)$ design admitting a block-transitive automorphism $G$ and $X\unlhd G\leq \Aut(X)$ with $X = PSL(n,q)$ and $n\geq2$. Suppose that  $\alpha\in\mathcal{P}$, $G_\alpha$ is  the point-stabilizer of $G$. Then $t=2$ and one of the  following holds:}
		 	 	\begin{enumerate}
		 	 	\item[{\rm(1)}]\textup{$X\cap G_\alpha\cong D_{14}$, $(n,q,v,k)=(2,8,36,6)$;}
		 		\item[{\rm(2)}]\textup{$X\cap G_\alpha\cong \mkern-1mu \raisebox{0.5ex}{$\hat{\ }$} \mkern-3mu (q^2+q+1):3$, $(n,q,v,k)=(3,3,144,12)$;}
		 		\item[{\rm(3)}]\textup{$X\cap G_\alpha\cong\mkern-1mu \raisebox{0.5ex}{$\hat{\ }$} \mkern-3mu[q^3]:GL(3,q)$, $(n,q,v,k)=(4,7,400,20)$;}
		 		\item[{\rm(4)}]\textup{$X\cap G_\alpha\cong \mkern-1mu \raisebox{0.5ex}{$\hat{\ }$} \mkern-3mu[q^4]:GL(4,q)$, $(n,q,v,k)=(5,3,121,11)$.}
		 	\end{enumerate}
		 	
\end{corollary} 
We now undertake a more detailed examination of the cases described in Theorem \ref{th1.2}. By introducing the constraint that $\lambda\mid k$, we derive the following result.
\begin{theorem}\rm\label{th13}
Let $\mathcal{D} = (\mathcal{P},\mathcal{B})$ be a non-trivial $t$-$(k^2,k,\lambda)$ design with $\lambda\mid k$, admitting a block-transitive automorphism $G$ and $X\unlhd G\leq \Aut(X)$, where $X = PSL(n,q).$ Then $X=PSL(3,3)$ and $\lambda\in\{3,6,12\}$.
\end{theorem}

\section{Preliminaries}

In this section, the notation and definitions in both design theory and group theory are standard and can be found in \cite{MR0827219,MR1409812,MR0812053}. Now we state some valuable findings from both design theory and group theory.

\begin{lemma}\label{21}\textup{ Let   $\mathcal{D}=(\mathcal{P},\mathcal{B})$ be a $t$-$(v ,k, \lam)$ design, then $\D$ is a $s$-$(v ,k, \lam_s)$ design for any $s$ with $1\leq s\leq t,$ and $$\lam_s=\lam\frac{\tbinom{v-s}{t-s}}{\tbinom{k-s}{t-s}}.$$}
\end{lemma}
  The following notation is adopted throughout: $\lambda_0$ represents the total number of blocks, commonly referred to as $b$. The parameter $\lambda_1$ indicates the number of blocks containing a specific point, which is typically denoted by $\gamma$.

	 	\begin{lemma}\label{GuanL1}{\rm \cite[Corollary 2.1]{MR4830077}}  \textup{Let $\mathcal{D} = (\mathcal{P},\mathcal{B})$ be a $t$-$(v,k,\lambda)$ design with $G\leq \Aut(\mathcal{D})$,  and let $d$ be a non-trivial subdegree of $G$. If $G$ is block-transitive and $v=k^2$, then $k+1$ $\mid$ $d$.}
	 \end{lemma}

	 \begin{lemma}\label{GuanL2}{\rm \cite[Lemma 4.1]{MR4830077}} \textup{Let $\mathcal{D}$ be a $t$-$(k^2,k,\lambda)$ design admitting a block-transitive automorphism group $G$.	If $\alpha$ is a point of $\mathcal{P}$, then $\frac{k+1}{(k+1,| \Out(X) |)}$ divides $|X_\alpha|.$ } 		
	 \end{lemma}
	 In this paper, $(x,y)$ denotes the greatest common divisor of positive integers $x$ and $y$. Combining Lemmas \ref{GuanL1} and \ref{GuanL2}, the following result is obvious.
	 \begin{corollary}\label{GuanC1}
	 	\textup{Let $\mathcal{D} = (\mathcal{P},\mathcal{B})$ be a $t$-$(k^2,k,\lambda)$ design and $G\leq  \Aut({\mathcal D})$ is block-transitive,   and $\alpha \in \mathcal{P}$, then $k+1 \mid (v-1,|\Out(X)||X_\alpha|)$.}
	 \end{corollary}
	 	 For a given positive integer $n$ and a prime divisor $p$ of $n$, we denote the $p$-part of $n$ by $n_p$, that is to say, $n_p=p^u$ with $p^u\mid n$ but $p^{u+1}\nmid n$ where $u$ is a positive integer.
	 \begin{lemma}\label{GuanC2}{\rm \cite[Corollary 2.11]{Chen}}
	 \textup{Suppose that $\mathcal{D}=(\mathcal{P},\mathcal{B})$ be a block-transitive $t$-$(v,k,\lambda)$ design with $v=k^2$, $G\leq$ Aut($\mathcal{D}$), $\alpha\in\mathcal{P}$, and $X=\Soc(G)$ be a simple group of Lie type in characteristic $p$. If the point stabilizer $G_\alpha$ is not a parabolic subgroup of $G$, then $(p,v-1)=1,$ and $|G|<|G_\alpha||G_\alpha|_{p^{\prime}}^2$.  In particular, 	
	  \begin{equation}
	 		\label{o}
	 		|X|<|\Out(X)|^2|X\cap G_\alpha||X\cap G_\alpha|_{p^{\prime}}^2.
	 \end{equation}}
	 \end{lemma}

 \begin{lemma}\label{p}{\rm \cite[Lemma 3.9]{MR0907231}} 
	{\rm Let $X$ be a group of Lie type in characteristic $p$, acting on the set of cosets of a maximal parabolic subgroup, and $X$ is not $PSL(n,q)$, $P\Omega^{+}(2m,q)$ with $m$ odd, nor $E_6(q)$, then there is a unique subdegree which is a power of $p$.}
\end{lemma}
  \begin{lemma}\label{bound}{\rm \cite[Lemma 4.2, Corollary 4.3]{MR3272379}}
	\textup{An upper bound and a lower bound of the order of some finite classical groups of Lie type are listed in Table \ref{Table:bound}.}
\end{lemma}
 \begin{table}[h]
	\centering
	\caption{\centering Bounds for the orders of some classical groups $G$\label{Table:bound}}
	\begin{tabular}{lllllll}
		\toprule
		$G$&A lower bound of $|G|$&An upper bound of $G$&	conditions\\	
		\midrule
		$GL(n,q)$&$q^{n^2-1}$&$q^{n^2}$&$n\geq2$\\
		$PSL(n,q)$&$q^{n^2-2}$&$(1-q^{-2})q^{n^2-1}$&$n\geq2$\\
		$GU(n,q)$&$q^{n^2-1}$&$q^{n^2+1}$&$n\geq3$\\
		$PSU(n,q)$&$(1-q^{-1})q^{n^2-2}$&$q^{n^2-1}$&$n\geq3$\\
		$Sp(n,q)$&$q^{[n(n-1)-2]/2}$&$q^{n(n+1)/2}$&$n\geq4$\\
		$PSp(n,q)$&$2^{-1}\cdot(2,q-1)^{-1}\cdot q^{n(n+1)/2}$&$q^{n(n+1)/2}$&$n\leq4$\\
		$SO^\epsilon(n,q)$&$q^{[n(n-1)-2]/2}$&$(2,q)\cdot q^{n(n-1)/2}$&$n\geq5$ and $\epsilon\in\{+,-,\circ\}$\\
		$P\Omega^\epsilon(n,q)$&$4^{-1}\cdot(2,n)^{-1}\cdot q^{n(n-1)/2}$&$q^{n(n-1)/2}$&$n\geq7$ and $\epsilon\in\{+,-,\circ\}$\\
		\bottomrule
	\end{tabular}
\end{table}

\section{Proof of the main result}
In this section, suppose that $\mathcal{D}$ is a non-trivial $(v,k,\lambda)$ design with $v=k^2$ and $G$ is a block-transitive almost simple automorphism group of $\mathcal{D}$ with socle $X=PSL(n,q)$, where $q=p^f$ and $n\geq3$. Let $\alpha\in\mathcal{P}$, $H:=G_\alpha$ and $H_0:=X\cap H$. As $G$ is point-primitive(\cite[Lemma 3.1]{MR4830077}) on $\mathcal{P}$, then the point-stabilizer $H$ is maximal in $G$(\cite[Corollary 1.5A]{MR1409812}). According to Aschbacher's Theorem(\cite{MR0746539}), $H$ is a $\mathcal{C}_i$-subgroup with the additional subgroups $\mathcal{C}_1^{\prime}$ or a $\mathcal{S}$-subgroup, where $i=1,\cdots,8$. The classification of these subgroups follows the framework established in \cite{MR1057341}. A rough description of the $\mathcal{C}_i$-subgroup is given in Table \ref{Table:C}.
\begin{table}[h]
	\centering
	\caption{The geometric subgroup collections\label{Table:C}}
	\begin{tabular}{llcllcl}
		\toprule
		Class & Rough description\\
		\midrule
		$\mathcal{C}_1$&Stabilizers of subspaces of $V$\\
		$\mathcal{C}_2$&Stabilizers of decompositions $V=\oplus_{i=1}^aV_i$, where dim$V_i=e$\\
		$\mathcal{C}_3$&Stabilizers of prime index extension fields of $\mathbb{F}$\\
		$\mathcal{C}_4$&Stabilizers of decompositions $V=V_1\otimes V_2$\\
		$\mathcal{C}_5$&Stabilizers of prime index subfields of $\mathbb{F}$\\
		$\mathcal{C}_6$&Normalisers of symplectic-type $r$-groups in absolutely irreducible representations\\
		$\mathcal{C}_7$&Stabilizers of decompositions $V=\otimes_{j=1}^{i}V_j$, where dim$V_j=\ell$\\
		$\mathcal{C}_8$&Stabilizers of non-degenerate forms on $V$\\
		\bottomrule
	\end{tabular}
\end{table}

 Moreover, by \cite[Lemma 2.2]{MR3470950}, we have 
\begin{equation}
	\label{e1}
	v=\frac{|X|}{|H_0|}= \frac{\prod \limits_{j=0}^{n-1}(q^n - q^j)/(q-1)}{(n,q-1)|H_0|}
\end{equation}
\noindent
and
\begin{equation}
	\label{e2}
	|\Out(X)|=2f(n,q-1)
\end{equation}
by \cite[Table 5.1.A]{MR1057341}.
Now we will consider the possible structure of  $H$ one by one.

\begin{lemma}\rm\label{C1}
	If $H$ is a $\mathcal{C}_1$-subgroup, then $H\cong P_1$, $t=2$, and one of the following holds:
		\begin{enumerate}
		\item[{\rm(1)}]\textup{$H_0 \cong \mkern-1mu \raisebox{0.5ex}{$\hat{\ }$} \mkern-3mu[q^3]:GL(3,q)$, $(n,q,v,k)=(4,7,400,20)$;}
		\item[{\rm(2)}]\textup{$H_0\cong \mkern-1mu \raisebox{0.5ex}{$\hat{\ }$} \mkern-3mu[q^4]:GL(4,q)$, $(n,q,v,k)=(5,3,121,11)$.}
	\end{enumerate}
\end{lemma}
\begin{proof}
 Let $H$ be a $\mathcal{C}_1$-subgroup. In this case, $H\cong P_i$ stabilizes an $i$-dimension subspace of $V$, with $i\leq n/2$. According to (\ref{e1}), we conclude that 

\begin{equation}
	\label{v}
	v=\frac{(q^n-1)\cdots (q^{n-i+1}-1)}{(q^i-1)\cdots (q-1))}>q^{i(n-i)}.
\end{equation}

If $i=1$, then $v=(q^n-1)/(q-1)$. Suppose first that $n=3$, then $v=(q^2+q+1)$, and hence $q^2<v<(q+1)^2$, which is impossible for $v=k^2.$  Thus $n>3$. According to \cite[A8.1]{MR1259738}, we have that $n=4,q=7,v=400$ or $n=5,q=3,v=121$ if $v$ is a perfect square. Thus, $H_0\cong \mkern-1mu \raisebox{0.5ex}{$\hat{\ }$} \mkern-3mu[q^3]:GL(3,q)$ or $H_0\cong \mkern-1mu \raisebox{0.5ex}{$\hat{\ }$} \mkern-3mu[q^4]:GL(4,q)$ respectively from \cite[Tables 8.8 and 8.18]{MR3098485}. If $(n,q,k)=(4,7,20)$, we have  $b=\lambda_3\cdot17\cdot20\cdot199/3$ for $t\geq3$ by Lemma \ref{21}. Note that $b\mid |G|$ for  $G$ is block-transitive, a contradiction. Hence $t=2$. Similarly, we have $t=2$ if $(n,q,k)=(5,3,11)$.

Now suppose that  $1< i\leq n-1$, by \cite[p.338]{MR1940339}, we get a subdegree 
\begin{equation}
	\label{d1}
	d=q(q^i-1)(q^{n-i}-1)/(q-1)^2.
\end{equation}
Then by Lemma \ref{GuanL1}, we have $$k+1\mid q(q^i-1)(q^{n-i}-1)/(q-1)^2.$$
According to the facts $2i\leq n$ and $q/2\leq(q-1)^2$ for $q\geq2$, we conclude that
$$q^{i(n-i)}<v=k^2<\frac{q^2(q^i-1)^2(q^{n-i}-1)^2}{(q-1)^4}<\frac{q^2\cdot q^{2i}\cdot q^{2n-2i}}{q^2/4}<q^{2n+2}.$$
It follows that $n(i-2)<i^2+2,$ and hence $i^2-4i-2<0$, which is true only for $i\in\{2,3,4\}.$

\medskip
\noindent
{\bf Case 1:} $i=4$. 

\medskip
\noindent
{\bf Subcase 1:} Assume first that $n\geq9$, then $$q^{4(n-4)}<q^2(q^4-1)^2(q^{n-4}-1)^2/(q-1)^4<4q^{2n},$$
which implies that $q^2\leq q^{2n-16}<4$, a contradiction.

\medskip
\noindent
{\bf Subcase 2:} Now we consider  the subcase $n=8$, we have $$v=\frac{(q^8-1)(q^7-1)(q^6-1)(q^5-1)}{(q^4-1)(q^3-1)(q^2-1)(q-1)}$$
by (\ref{v}), and
$$k+1\mid q(q^4-1)^2/(q-1)^2.$$ Thus, $$q^{16}<v=k^2<q^2(q^4-1)^4/(q-1)^4,$$
which implies $q=2$. However, in this case, $v=3\cdot17\cdot31\cdot127$ is not a perfect square. 

\medskip
\noindent
{\bf Case 2:} $i=3.$

Since $i=3$, $q^{3(n-3)}<4\cdot q^{2n}$, it follows that $q^{n-9}<4$. Thus $n\in\{6,7,8,9,10\}.$

\medskip
\noindent
{\bf Subcase 1:} $q=2.$

For $n=6,7,8,9$ or $10$, all possible values of $v$, as determined by formula (\ref{v}), are listed in Table \ref{Table:1}. As demonstrated in Table \ref{Table:1}, $v$ is not a perfect square for all considered cases, a contradiction.
\begin{table}[h]
	\centering
	\caption{Possible values of $n$ and $v$ when $q = 2$\label{Table:1}}
	\begin{tabular}{llll}
		\toprule
		$n$   & $v$&$n$ & $v$ \\
		\midrule
		$6$&$3^2\cdot5\cdot31$&$7$&$7\cdot47\cdot159$\\
		$8$&$3^2\cdot5\cdot17\cdot127$&$9$&$5\cdot17\cdot73\cdot127$\\
		$10$&$3\cdot5\cdot11\cdot17\cdot31\cdot73$\\
		\bottomrule
	\end{tabular}
\end{table}

\noindent
{\bf Subcase 2:} $q>2.$

Suppose that $n=6$, by (\ref{v}) and (\ref{d1}), we know that
\begin{align*}
	(v-1,d)&=\left(\frac{(q^6-1)(q^5-1)(q^4-1)}{(q^3-1)(q^2-1)(q-1)}-1,\frac{q(q^3-1)(q^3-1)}{(q-1)^2}\right)\\
	      &=\left(q^4-2q^3-q^2-2q,\frac{q(q^3-1)(q^3-1)}{(q-1)^2}\right),
\end{align*}
and hence $$k+1\mid q^4-2q^3-q^2-2q$$ by Lemma \ref{GuanL1}, which implies that
$v= q^{10}+q^9+2q^8+3q^7+3q^6+4q^5+3q^4+3q^3+2q^2+q+1<q^8$, which is impossible.

\medskip
\noindent
{\bf Case 3:} $i=2.$
Here we have that $v=\frac{(q^n-1)(q^{n-1}-1)}{(q^2-1)(q-1)}$ by (\ref{v}), and $G$ is rank 3 with non-trivial subdegrees(\cite[p.338]{MR1940339}): 
$$d_1=\frac{q(q+1)(q^{n-2}-1)}{q-1}$$ 
and
$$d_2=\frac{q^4(q^{n-2}-1)(q^{n-3}-1)}{(q^2-1)(q-1)}.$$ 

\medskip
\noindent
{\bf Subcase 1:} Suppose first that $n$ is even.  Then $q+1$ is coprime to $(q^{n-3}-1)/(q^2-1)$, we conclude that $k+1\mid f(q)$ by Lemma \ref{GuanL1}, where $f(q)=(d_1,d_2)=q(q^{n-2}-1)/(q^2-1)$. Thus $$v=\frac{(q^n-1)(q^{n-1}-1)}{(q^2-1)(q-1)}=k^2<\frac{q^2(q^{n-2}-1)^2}{(q^2-1)^2},$$
it follows that $$(q^n-1)(q^{n-1}-1)<q^2(q^{n-2}-1)^2<q^{2n-3},$$ which is impossible.

\medskip
\noindent
{\bf Subcase 2:} Assume that $n$ is odd. By Lemma \ref{GuanL1}, we have that $k+1$ divides $$(d_1,d_2)\mid\delta \cdot q(q^{n-2}-1)/(q-1),$$ where $\delta=(q+1,n-3)$.

First we consider $n=3$, then $v=q^2+q+1$. Note that $q^2<q^2+q+1<(q+1)^2$, which implies that $v$ cannot be a perfect square.

Now we suppose that $n=5,$ then $v=(q^2+1)(q^4+q^3+q^2+q+1)$
and 
$k+1\mid2q(q^2+q+1).$
Let $m(k+1)=2q(q^2+q+1)$, where $m$ is a positive integer. Then $$(q^2+1)(q^4+q^3+q^2+q+1)<\frac{4q^2(q^2+q+1)^2}{m^2},$$
which implies that
\begin{align*}
	m^2&<\frac{4q^2(q^2+q+1)^2}{(q^2+1)(q^4+q^3+q^2+q+1)}<\frac{4q^2(q^2+q+1)^2}{q^6}=4(1+\frac{2}{q}+\frac{3}{q^2}+\frac{2}{q^3}+\frac{1}{q^4})\\
	&<4(1+\frac{2}{2}+\frac{3}{2^2}+\frac{2}{2^3}+\frac{1}{2^4})=12+\frac{1}{4}.
\end{align*}
Thus we get that $m=1,2,3.$

If $m=1$ or $2$, then $k=2q(q^2+q+1)-1$ or $q(q^2+q+1)-1$, which contradicts the basic equation $v=k^2.$

If $m=3$, then $k+1=\frac{2}{3}q(q^2+q+1)$. Since $v<(k+1)^2$, we have $$(q^2+1)(q^4+q^3+q^2+q+1)<\frac{4}{9}q^2(q^2+q+1)^2.$$
It follows that $$5q^6+q^5+6q^4+10q^3+14q^2+9q+9<0,$$
which is impossible.

Therefore, we only need to consider the case where $n\geq7$ now. Here 
	$$v=(q^{n-1}+q^{n-2}+\cdots+q^2+q+1)(q^{n-3}+q^{n-5}+\cdots+q^4+q^2+1).$$
Suppose that $\xi= \frac{q(q^{n-2}-1)}{q-1}=q(q^{n-3}+q^{n-4}+\cdots+q+1)$, then $k+1\mid\xi\delta.$ Thus, there exists a positive integer $l$ such that 
\begin{equation}
	\label{l}
	l(k+1)=\xi\delta,
\end{equation}
it follows that 
 $v<\frac{\xi^2\delta^2}{l^2}$. Moreover, $(q+1,n-3)<2q.$ Hence

\begin{align*}
	l^2&<\frac{\xi^2\delta^2}{v}=\frac{q^2(q^{n-3}+q^{n-4}\cdots+q+1)^2\cdot(q+1,n-3)^2}{(q^{n-1}+q^{n-2}+\cdots+q^2+q+1)(q^{n-3}+q^{n-5}+\cdots+q^4+q^2+1)}\\
	&<\frac{q^2(2q^{n-3})^2\cdot(2q)^2}{q^{2n-4}}=16q^2,
\end{align*}
which implies that 
\begin{equation}
	\label{l33}
	l<4q.
\end{equation}
Note  the following equalities
$$\frac{v-1}{\xi}=q^{n-2}+q^{n-4}+\cdots+q^3+q+1,$$
and $v-1=(k+1)(k-1)$, we have 
$$k+1=\frac{v-1}{k+1}+2=\frac{lq^{n-2}+lq^{n-4}+\cdots+lq^3+lq+l+2\delta}{\delta}$$ by (\ref{l}).
Hence $(k+1)\delta=lq^{n-2}+lq^{n-4}+\cdots+lq^3+lq+l+2\delta.$ It is straightforward to obtain
\begin{equation}
	\label{l1}
	\left((k+1)\delta,q\right)=(l+2\delta,q).
\end{equation}
Moreover, $\left((k+1)\delta,\frac{q^{n-2}-1}{q-1}\right)$ divides $\left((k+1)\delta,\frac{l(q^{n-2}-1)}{q-1}\right)$ 
and
\begin{equation}
	\label{l2}
	 \left((k+1)\delta,\frac{l(q^{n-2}-1)}{q-1}\right) =\left(lq^{n-4}+lq^{n-6}+\cdots+lq+2l+2\delta,(2l+2\delta)q+l+2\delta\right).
\end{equation}

\noindent
As $k+1\mid \xi\delta$, we have that  $(k+1)\delta=\left((k+1)\delta,\xi\delta^2\right)$, which divides
\begin{equation}
	\label{l3}
 \left((k+1)\delta,\delta^2\right)\cdot \left((k+1)\delta,\xi\right).
\end{equation}
  Since $\left((k+1)\delta,\xi\right)$ divides
$$
  \left((k+1)\delta,q\right)\cdot\left((k+1)\delta,\frac{q^{n-2}-1}{q-1}\right),
$$
 according to (\ref{l1}), (\ref{l2}) and (\ref{l3}), we have that 
 \begin{equation}
 	\label{l5}
 	(k+1)\delta\mid\delta^2 q[(2l+2\delta)q+l+2\delta],
 \end{equation}
and hence
\begin{equation}
	(k+1)\delta\leq \delta^2q[(2l+2\delta)+l+2\delta].
\end{equation}
It is noteworthy that
\begin{align*}
	q^{n-2}+q^{n-4}+\cdots+q^3+q+1&\leq (k+1)\delta\\
	&\leq\delta^2q[(2l+2\delta)+l+2\delta]\\
	&<(q+1)^2q[(8q+2q+2)q+4q+2q+2]\\
	&=10q^5+28q^4+28q^3+12q^2+2q,
\end{align*}
by (\ref{l33}). If $n\geq13$, we have 
 $10q^5+28q^4+28q^3+12q^2+20q<q^{11}+q^{9}+\cdots+q^3+q+1$, which leads a contradiction. Thus, $n=7,9$ or $11$. 
 Note that 
\begin{align*}
	q^{n-2}+q^{n-4}+\cdots+q^3+q+1&\leq (k+1)\delta\\
	&\leq\delta^2q[(2l+2\delta)+l+2\delta]\\
	&<(n-3)^2q[(8q+n-3)q+4q+2(n-3)]\\
	&=(n-3)^2[8q^3+(n+1)q^2+2(n-3)q)].
\end{align*}

If $n=7$, then  we have that $q^5+q^3+q+1<128q^3+128q^2+128q,$ which implies that  $q\in\{2,3,4,5,7,8,9,11\}.$
Similarly, we have $q\in\{2,3,4\}$ if $n=9$ and $q\in\{2,3\}$ if $n=11.$

Now we consider the case when $n=11$ and $q=2$. Here $\delta=(3,11-3)=1$ and $1\leq l\leq8$ by (\ref{l33}).  Then $(k+1)\delta=22144l+8$ and $\delta^2q[(2l+2\delta)+l+2\delta]=336l+1536$, which is impossible by (\ref{l1}). The rest cases can be ruled out similarly.
\end{proof}

\begin{lemma}\rm\label{C_1'}
	The subgroup $H$ cannot be a $\mathcal{C}_1^\prime$-subgroup.
\end{lemma}
\begin{proof}
 Let $H$ be a $\mathcal{C}_1^\prime$-subgroup. In this case, $G$ contains a graph automorphism and the stabilizer $H$ stabilizes a pair $\{U,W\}$ of subspaces of dimension $i$ and $n-i$ respectively with $i<n/2.$ Thus, we have the following two possibilities:

\medskip
\noindent
{\bf Case 1:} Suppose first that $U$ is contained in $W$. Then we have that $v>q^{2ni-3i^2}$ by (\ref{e1})
 and Lemma \ref{bound}. According to Lemma \ref{p} and \cite[p.339]{MR1940339}, there is a subdegree with a power of $p$. On the other hand, if $p$ is odd, then the highest power of $p$ dividing $v-1$ is $q$, it is $2q$ if $q>2$ is even, and is most $2^{n-1}$ if $q=2$. Therefore we have $k+1\mid2q$ by Lemma \ref{GuanL1}, and hence $q^{2ni-3i^2}<4q^2.$
 It follows that $i^2-3<2ni-3i^2-3<0,$ which  implies that $i=1.$ Then we get that $2n-6<0,$ a contradiction.
 
 \medskip
\noindent
{\bf Case 2:} Now assume that $U\cap W=0$. By \cite[Proposition 4.1.4]{MR1057341}, we have that

$$|H_0|=\frac{|SL(i,q)|\cdot|SL(n-i,q)|\cdot(q-1)}{(n,q-1)}.$$
If $i$=1, we have that $v=q^{n-1}(q^n-1)/(q-1)$ by (\ref{e1}). From \cite[p.339]{MR1940339}, we get a subdegree $d=q^{n-2}(q^{n-1}-1)/(q-1)$. Let $f(q)=(d,v-1)=(q^{n-1}-1)/(q-1).$ Since $k+1$ divides $f(q)$ by Lemma \ref{GuanL1}, there exists a positive integer $m$ such that $m\cdot(k+1)=(q^{n-1}-1)/(q-1).$ Then $$m^2\cdot q^{n-1}(q^n-1)/(q-1)<(q^{n-1}-1)^2/(q-1)^2.$$
However, there are no solutions in the inequality. Hence $i\geq2$. According to Lemma $\ref{bound}$ and (\ref{e1}), we have that $v>q^{2i(n-i)}$, and
by \cite[p.340]{MR1940339}, there is a subdegree $d=2\cdot(q^i-1)(q^{n-i}-1)$. Then $k+1\mid 2\cdot(q^i-1)(q^{n-i}-1)$ by Lemma \ref{GuanL1}, which implies that 
$$q^{2i(n-i)}< 4\cdot(q^i-1)^2(q^{n-i}-1)^2,$$
and hence  $$q^{2n(i-1)-2i^2-1}<2.$$ Since $n<2i$, we obtain that ${2i^2-4i-1}<0.$ It implies that $i=1$, a contradiction. 
\end{proof}

\begin{lemma}\rm\label{C_2}
	The subgroup $H$ cannot be $\mathcal{C}_2$-subgroup.
\end{lemma}
\begin{proof}
Let $H$ be a $\mathcal{C}_2$-subgroup. 
In this case, $H$ preserves a partition $V=V_1\oplus V_2\oplus\cdots\oplus V_a,$ with each $V_i$ of the same dimension $e$, where $n=ae.$ By \cite[Proposition 4.2.9]{MR1057341}, we have that 
\begin{equation}
	\label{h1}
|H_0|=|\mkern-1mu \raisebox{0.5ex}{$\hat{\ }$} \mkern-3muSL(e,q)^a|\cdot|S_a|\cdot(q-1)^{a-1}.
\end{equation}
 Moreover, we get that $v>q^{n(n-e)}/(a!)$ by \cite[p.12]{MR4090499}.
First we consider the case where $a=n, e=1,$ then we have a subdegree  $d=2n(n-1)(q-1)$ by \cite[p.340]{MR1940339}. It follows from Lemma \ref{GuanL1} that  $$k+1\mid 2n(n-1)(q-1),$$ 
and hence $$k^2=v=\frac{q^{n(n-1)}}{n!}<4n^2(n-1)^2(q-1)^2.$$
Thus, $n=3$, $q\in\{2,3,4\}$ or $n=4, q=2.$ For each case, $v$ is not a perfect square.

Now let $e\geq2,$ in which case, by \cite[p.340]{MR1940339}, there is a subdegree $$d=a(a-1)(q^{e}-1)^2/(q-1).$$ According to Lemma \ref{GuanL1} and the fact $(q-1)^2\geq\frac{q}{2}$, we conclude that $$q^{n(n-e)}/(a!)<v<a^2(a-1)^2(q^{e}-1)^4/(q-1)^2<2a^4q^{4e-1},$$
which implies that $$q^{n^2-ne-4e+1}<2a!a^4.$$
Note that $a!a^4<a^{a+4}$, then $$n^2-ne-4e+1\leq \log_2q^{n^2-ne-4e+1}<(a+4)\cdot\log_2a+1<(a+4)\frac{2a}{3}+1$$ for $\log_2a<\frac{2}{3}a.$ As $n=ae\geq3$, we obtain $$(3e^2-2)a^2-(3e^2+8)a-12e<0,$$ 
and which forces the pair $(a,e)$ mush be either $(2,2)$ or $(2,3)$.

If $(a,e)=(2,2)$, then $n=4$ and $d=2(q^2-1)^2/(q-1)=2(q+1)(q^2-1).$ It follows that $q^8/2<4(q+1)^2(q^2-1)$, thus, $q=2,3$, and hence $v=2^3\cdot5\cdot7$ or $3^4\cdot5\cdot13$, a contradiction.
Similarly, $(a,e)=(2,3)$ can be ruled out.
\end{proof}
 
 \begin{lemma}\rm\label{C3}
 	If $H$ is a $\mathcal{C}_3$-subgroup, then $H_0\cong \mkern-1mu \raisebox{0.5ex}{$\hat{\ }$} \mkern-3mu(q^2+q+1):3$, $t=2$ and $(n,q,v,k)=(3,3,144,12).$
  \end{lemma}
\begin{proof}
 Let $H$ be a $\mathcal{C}_3$-subgroup. Then $H$ is an extension field subgroup and by \cite[Proposition 4.3.6]{MR1057341}, we get 
$|H_0|=|\mkern-1mu \raisebox{0.5ex}{$\hat{\ }$} \mkern-3muSL(i,q^\theta)|\cdot(q^\theta-1)\cdot \theta/(q-1)$, with $n=i\theta$ and $\theta$ prime. If $\theta=2$, then by (\ref{e1}) and Lemma \ref{bound}, we have that
$$q^{2i^2-4}<q^{i^2}(q^{2i-1}-1)(q^{2i-3}-1)\cdots(q-1)/2=v.$$

Assume first that $n\geq8$, there is a subdegree $d=(q^{2i}-1)(q^{2i-2}-1)$ by \cite[p.341]{MR1940339}. Then $k+1\mid f(q)$ by Lemma \ref{GuanL1}, where $f(q)=(v-1,d).$ Note that $v-1$ is coprime to $(q^{i-1}-1)$ if $i$ is even and its coprime to $q^i-1$ if $i$ is odd. Therefore, $f(q)=(v-1,d)\leq(q^{2i}-1)(q^{i-1}+1)$, it follows that $q^{2i^2-4}<(q^{2i}-1)^2(q^{i-1}+1)^2.$ Hence $i^2-3i-1<0,$ which is impossible as $i\geq4.$  

Now we consider $n=6$, then $v=q^9(q^5-1)(q^3-1)(q-1)/2,$ and $f(q)\leq(q^6-1)(q^2+1),$ which implies that
$$q^9(q^5-1)(q^3-1)(q-1)/2<(q^6-1)^(q^2+1)^2.$$
Thus we have that $q=2,3.$ 
If $q=2$ or $3$, $v=2^8\cdot7\cdot31$ or $2^4\cdot3^9\cdot11^2$ recpectively, a contradiction.

Next we assume that $n=4$, then $v=q^4(q^3-1)(q-1)/2$ and $d=(q^4-1)(q^2-1).$
Here we have that $(v-1,d)=(v-1,(q^3+q^2+q+1)(q+1)).$
According to 
 $$q^4(q^3-1)(q-1)-2=(q^3+q^2+q+1)(q+1)\cdot(q^4-3q^3+4q^2-5q+8)-(11q^3+10q^2+11q+10)$$
and
$$121(q^3+q^2+q+1)(q+1)=(11q^3+10q^2+11q+10)(11q+12)+(q^2+1),$$
we conclude that $k+1\mid q^2+1$ by Lemma \ref{GuanL1}. Thus 
$$q^4(q^3-1)(q-1)/2<(q^2+1)^2,$$ which is impossible.

Now we consider $\theta\geq3,$ here we have that
\begin{equation}
	\label{r1}
	q^{n^2-2}<4f^2\cdot q^{(2i+1)n}\cdot \theta^3
\end{equation}
by Lemmas \ref{bound} and (\ref{o}). Note the fact that $4f^2\leq q^2$. It follows that 
\begin{equation}
	\label{r2}
	q^{n^2-(2i+1)n-4}<\theta^3. 
\end{equation}
Thus $$n^2-(2i+1)n-4\leq(n^2-(2i+1)n-4)\cdot \log_2q<3\cdot \log_2\theta.$$
Since $\log_2\theta\leq2\theta/3$ and $n=i\theta$, we get that
$i^2\theta^2-(2i+1)i\theta-2\theta-4<0.$ According to $i\theta\geq5,$ straightforward computation shows that the last inequality holds only for pairs $(i,\theta)=(1,3),(1,4),(1,5)$ or $(2,3).$

\medskip
\noindent
{\bf Case 1:} If $(i,\theta)=(1,5)$, according to (\ref{r2}), we have $q=2$. In which case, $v=2^{10}\cdot3^2\cdot7$, which is not a perfect square.

\medskip
\noindent
{\bf Case 2:} If $(i,\theta)=(2,3)$,  by (\ref{r2}), we have $q=2,3,4,$ or $5.$ Here assume $q=5$, then $5^{34}<4\cdot5^{30}\cdot27$ by (\ref{r1}), a contradiction. Consequently, the only possible values for $q$ are 2, 3, and 4.

If $q=2$, then $$v=|PSL(6,2)|/(|\mkern-1mu \raisebox{0.5ex}{$\hat{\ }$} \mkern-3muSL(2,8)|\cdot(8-1)\cdot3)=2^{14}\cdot3\cdot5\cdot31,$$
which is not a perfect square. 

%If $q=3$, then
%$$v=|PSL(6,3)|/(|\mkern-1mu \raisebox{0.5ex}{$\hat{\ }$} \mkern-3muSL(2,27)|\cdot(27-1)\cdot3)/2=2^{9}\cdot3^{11}\cdot5\cdot11^2,$$ a contradiction.

%If $q=4$, then
%$$v=|PSL(6,4)|/(|\mkern-1mu \raisebox{0.5ex}{$\hat{\ }$} \mkern-3muSL(2,64)|\cdot(64-1)\cdot3)/3=2^{8}\cdot3^{12}\cdot5\cdot11^2=2^{24}\cdot3^4\cdot5\cdot7\cdot11\cdot31,$$
%which is not a perfect square.
The cases $q$ = 3 and 4 can all be excluded by the same calculation.

\medskip
\noindent
{\bf Case 3:} If $(i,\theta)=(1,3)$, then
 $n=3$ and $H_0\cong \mkern-1mu \raisebox{0.5ex}{$\hat{\ }$} \mkern-3mu(q^2+q+1):3.$
Thus, $|H_0|=\frac{3(q^2+q+1)}{(3,q-1)}$, and we have that
	$$v=\frac{q^3(q^2-1)(q^3-1)}{(3,q-1)}\cdot\frac{(3,q-1)}{3(q^2+q+1)}=\frac{q^3(q-1)^2(q+1)}{3}.$$
	Note that $|\Out(X)||H_0| = 2f(3,q-1)\cdot\frac{3(q^2+q+1)}{(3,q-1)}=6f(q^2+q+1).$ Thus, $$\frac{q^3(q-1)^2(q+1)}{3}=v=k^2\leq(6f(q^2+q+1)-1)^2<36f^2(q^2+q+1)^2$$ by Lemma \ref{GuanL2}, and hence $$q^{3}(q-1)^2(q+1)\leq108f^2(q^{2}+q+1)^2.$$
	This inequality holds when 
	\begin{equation}
		\label{e}
q\in\{2,3,4,5,6,7,8,9,11,16,27,32\}.
	\end{equation}
	Then, for each $q$, the possible values of $v$ are listed in Table \ref{Table:e}.
	\begin{table}[h]
	\centering
	\caption{Possible values for $v$ with $q=p^f$ \label{Table:e}}
	\begin{tabular}{llll}
		\toprule
		$q$   & $v$	&$q$   & $v$ \\
		\midrule
		$2$&$2^3$&$3$&$2^4\cdot3^2$\\
		$4$&$2^6\cdot3\cdot5$&$5$&$2^5\cdot3^5$\\
		$7$&$2^5\cdot3\cdot7^3$&$8$&$2^9\cdot3\cdot7^2$\\
		$9$&$2^7\cdot3^5\cdot5$&$11$&$2^4\cdot5^2\cdot11^3$\\
		$16$&$2^{12}\cdot3\cdot5^2\cdot17$&$27$&$2^{2}\cdot3^{12}\cdot7\cdot13^2$\\
		$32$&$2^{15}\cdot11\cdot31^2$\\
		\bottomrule
	\end{tabular}
	\end{table}
	
	As shown in Table \ref{Table:e}, $v$ is a perfect square if and only if $q = 3$. For this case, we have $k=12$ and $v=144$.

	Now we consider $q=3$, then $X=PSL(3,3)$, and $H_0\cong \mkern-1mu \raisebox{0.5ex}{$\hat{\ }$} \mkern-3mu (q^2+q+1):3$ with $q=3.$ Thus $G=PSL(3,3)$ or $PSL(3,3):2$, and the order of $G$ is 5616 or 11232 respectively. If $t\geq3$, then
	$$b=\lambda_3\frac{v(v-1)(v-2)}{k(k-1)(k-2)}\cdot=\lambda_3\cdot12\cdot13\cdot71/5,$$which get a contradiction for $b\mid|G|$. Thus $t=2.$
	
	\medskip
	\noindent
{\bf Case 4:} If $(i,\theta)=(1,4)$, then
$n=4$ and $|H_0|=\frac{4(q^3+q^2+q+1)}{(4,q-1)}$, and we have that $$v=\frac{(q^4-1)(q^4-q)(q^4-q^2)(q^4-q^3)}{(4,q-1)}\cdot\frac{(4,q-1)}{3(q^3+q^2+q+1)}=\frac{q^6(q-1)^2(q^2-1)(q^3-1)}{3}$$
and
$$|\Out(X)||H_0|=2f\cdot(4,q)\cdot\frac{4(q^3+q^2+q+1)}{(4,q-1)}=8f(q^3+q^2+q+1).$$
 By Lemma \ref{GuanL1}, we get that $$\frac{q^6(q-1)^2(q^2-1)(q^3-1)}{3}<64f^2(q^3+q^2+q+1)^2,$$
it follows that $q=2$ or $3$, and hence $v=2^6\cdot7$ or $2^6\cdot3^5\cdot13,$ a contradiction. 
\end{proof}

\begin{lemma}\rm\label{C4}
	The subgroup $H$ cannot be a $\mathcal{C}_4$-subgroup.
\end{lemma}
\begin{proof}
Let $H$ be a $\mathcal{C}_4$-subgroup. Then $H$ is the stabilizer of a tensor product of two non-singular subspaces of dimensions $1<i<n/i$, which implies $i<\sqrt{n}.$ As $n/i$ is a positive integer, so $n\geq6$. By \cite[Proposition 4.4.10]{MR1057341}, we have that 
$$|H_0|=\frac{(i,n/i,q-1)}{(n,q-1)}\cdot|PGL(i,q)||PGL(n/i,q)|.$$

\noindent
According to Lemma \ref{bound}, we get the fact 
$$|H_0|\leq|PGL(i,q)||PGL(n/i,q)|<\frac{(1-q^{-1})q^{i^2}}{q-1}\cdot\frac{(1-q^{-1})q^{n^2/i^2}}{q-1}=q^{i^2+(n^2/i^2)-2}.$$
Here we fix $n$, and let $\Phi_1(i)=i^2+(n^2/i^2)-2=(i+n/i)^2-2n-2.$ Then $\Phi_1(i)$ is a decreasing function on the interval (2,$\sqrt{n}$), and hence $\Phi_1(i)<\Phi_1(2)=n^2/4+2.$ Thus $|H_0|<q^{n^2/4+2}$.
It follows that $$q^{n^2-2}<4f^2\cdot q^{n^2/4+4}\cdot\prod \limits_{j=1}^{i}(q^j - 1)^2\cdot\prod \limits_{j=1}^{n/i}(q^j - 1)^2$$ by (\ref{o}). Note the fact that $\prod \limits_{j=1}^{i}(q^j - 1)\cdot\prod \limits_{j=1}^{n/i}(q^j - 1)<q^{\frac{i^2+i+(n/i)^2+(n/i)}{2}}$. Suppose that 
$$\Phi_2(i)=i^2+i+(n/i)^2+(n/i)=(i+n/i)^2+(i+n/i)-2n.$$
Then $\Phi_2(i)$ is a decreasing function of $i$ on the interval $(2,\sqrt{n})$, and hence $$\Phi_2(i)\leq\Phi_2(2)=\frac{n^2}{4}+\frac{n}{2}+6.$$
Thus we have that $$q^{n^2-2}<4f^2\cdot q^{\frac{n^2}{2}+\frac{n}{2}+10}\leq q^{\frac{n^2}{2}+\frac{n}{2}+12},$$ which implies that $n^2-n-28\leq0,$ and so $i=2,n=4$ or $5$, a contradiction. 
\end{proof} 

\begin{lemma}\rm\label{C5}
	The subgroup $H$ cannot be a  $\mathcal{C}_5$-subgroup.
\end{lemma}
\begin{proof}
 Let $H$ be a $\mathcal{C}_5$-subgroup. Then $H$ is the stabilizer in $G$ of a subfield space and by \cite[Proposition 4.5.3 ]{MR1057341}, we get $$|H_0|=|\mkern-3mu \raisebox{0.5ex}{$\hat{\ }$} \mkern-1mu PGL(n,q_0)|\cdot(n,(q-1)/(q_0-1))$$  with $q=q_0^{u}$, $u$ prime. If $u\geq3$
, by Lemma \ref{bound} and (\ref{o}), then we get that $$q_0^{u(n^2-2)}<4f^2\cdot q_0^{2n^2+n-3}\cdot(n,q_0^u-1)^3.$$
Since $4f^2\leq q_0^{2u}$ and $n^3\leq q_0^{2n}$, we have that $u\cdot(n^2-4)<2n^2+3n-3$, according to the fact $n\geq3$, we conclude that 3$\cdot$$(n^2-4)<2n^2+3n-3,$ which implies that $n^2-3n-15<0.$ Thus $n=3$ or $4$.

If $n=3$, we have that $q_0^{7u}<108f^2\cdot q_0^{18}.$
Let $q_0=p^a$, then $f=au$.
Thus we have that $$p^{7au}<108(au)^2\cdot p^{18a}.$$
According to the inequality, we have  $(p,u,a)=(2,3,1),(2,3,2),(2,3,3),(2,3,4),(2,4,1),$
\noindent
(3,3,1), (3,3,2), (5,3,1) and (7,3,1). However, for each case, we have no suitable parameters as $v=k^2.$

If $n=4$, we have that $q_0^{14u}<108f^2\cdot q_0^{33}.$
Let $q_0=p^a$, then $f=au$.
Thus we have that $$p^{14au}<108(au)^2\cdot p^{33a},$$
which is true only when $(p,u,a)=(2,3,1)$. However, $v$ is not a perfect square for this case.

Now we consider $u=2$,  we have that $v>q_0^{n^2-3}$ by (\ref{e1}) and Lemma \ref{bound}. Moreover, we know that there is a subdegree $(q_0^n-1)(q_0^{n-1}-1)$ from  \cite[p.343]{MR1940339}. Then by Lemma \ref{GuanL1}, we get that $$q_0^{n^2-3}<(q_0^n-1)^2(q_0^{n-1}-1)^2<q_0^{4n-2}.$$
It follows that $n^2-4n-1<0$, which forces $n=3,4.$

If $n=3$, then $v=q_0^3(q_0^2+1)(q_0^3+1)/\zeta_1$, where $\zeta_1=(3,q_0+1)$.
Then $\zeta_1=1$ or $3$.

\noindent
{\bf Subcase 1:} Let $\zeta_1=1$, then $$v=q_0^3(q_0^2+1)(q_0^3+1)$$ and $$|\Out(X)||H_0|=2f\cdot q_0^3(q_0^2-1)(q_0^3-1),$$
thus, $$(v-1,|\Out(X)||H_0|)=(v-1,2f(q_0^2-1)(q_0^3-1)).$$
As $$q_0^3(q_0^2+1)(q_0^3+1)-1=(q_0^2-1)(q_0^3-1)(q_0^3+2q_0+2)+\phi_1(q_0)$$
and
$$2(q_0^2-1)(q_0^3-1)=\phi_1(q_0)(q_0-2)+\varphi_1(q_0),$$
where $\phi_1(q_0)=2q_0^4+4q_0^3+2q_0^2-2q_0-3$, $\varphi_1(q_0)=4q_0^3+4q_0^2-q_0-4,$ thus we have $$k+1\mid 2f(\phi_1(q_0),\varphi_1(q_0)),$$
by Corollary \ref{GuanC1},
it follows that 
$$q_0^8<k^2=q_0^3(q_0^2+1)(q_0^3+1)\leq(2f\varphi_1(q_0)-1)^2<64f^2q_0^4(q_0+1)^2.$$
Here $f=2a$, thus, 
\begin{equation}
	\label{lz}
	p^{4a}<256a^2(p^a+1)^2.
\end{equation}
%From the preceding inequality, we get that
%\begin{equation}
%	\label{312.1}
%	\begin{aligned}
%		p&=2, &a\leq6;\\
%		p&=3, &a\leq3;\\
%		p&=5, &a\leq2;\\
%		p&=7,11,13, &a=1.\\
%	\end{aligned}
%\end{equation}a
According to (\ref{lz}) and the fact $\zeta_1=1$, 
all possible values of 
$q_0$ and $v$  are listed in Table \ref{Table:312.1}. 
\begin{table}[h]
	\centering
	\caption{Possible values for $v$ with $q_0 = p^a$ \label{Table:312.1}}
	\begin{tabular}{llll}
		\toprule
		$q_0$   & $v$ 	&$q_0$   & $v$\\
		\midrule
		$3$&$2^3\cdot3^3\cdot5\cdot7$&$4$&$2^6\cdot5\cdot13\cdot17$\\
		$7$&$2^4\cdot5^2\cdot7^3\cdot43$&$9$&$2^2\cdot3^6\cdot5\cdot41\cdot73$\\
		$13$&$2^2\cdot5\cdot7\cdot13^3\cdot17\cdot157$&$16$&$2^{12}\cdot17\cdot241\cdot257$\\
		$25$&$2^2\cdot5^6\cdot13\cdot313\cdot601$&$27$&$2^3\cdot3^9\cdot5\cdot7\cdot19\cdot37\cdot73$\\
		$64$&$2^{18}\cdot5\cdot13\cdot17\cdot37\cdot109$\\
		\bottomrule
	\end{tabular}
\end{table}

\noindent
However, $v$ is not a perfect square for each case.

\noindent
{\bf Subcase 2:} Suppose that $\zeta_1=3$, then $v=q_0^3(q_0^2+1)(q_0^3+1)/3$
and
$$|\Out(X)||H_0|=6f\cdot q_0^3(q_0^2-1)(q_0^3-1).$$
Here
\begin{align*}
	(v-1,|\Out(X)||H_0|)
	&=\frac{1}{3}(q_0^3(q_0^2+1)(q_0^3+1)-3,18fq_0^3(q_0^2-1)(q_0^3-1))\\
	&=\frac{1}{3}(q_0^3(q_0^2+1)(q_0^3+1)-3,18f(q_0^2-1)(q_0^3-1)).
\end{align*}

\noindent
Since $$q_0^3(q_0^2+1)(q_0^3+1)-3=(q_0^2-1)(q_0^3-1)\cdot(q_0^3+2q_0+2)+\phi_2(q_0)$$
and
$$2(q_0^2-1)(q_0^3-1)=\phi_2(q_0)\cdot(q_0-2)+\varphi_2(q_0),$$
where $\phi_2(q_0)=q_0^4+4q_0^3+2q_0^2-2q_0-5$, $\varphi_2(q_0)=4q_0^3+4q_0^2-3q_0-8.$
According to Corollary \ref{GuanC1}, $$k+1\mid6f(\phi_2(q_0),\varphi_2(q_0)),$$
and hence $$k^2=q_0^3(q_0^2+1)(q_0^3+1)/3\leq(6f\varphi_2(q_0)-1)^2<576f^2\cdot q_0^4(q_0^2+1)^2.$$
It follows that $$(q_0^2+1)(q_0^3+1)<1728f^2q_0(q_0+1)^2,$$
thus
\begin{equation}
	\label{lx}
	(p^{2a}+1)(p^{3a}+1)<6912a^2p^a(p^a+1)^2.
\end{equation}

%Solving the above inequality, we have that 
%\begin{equation}
%	\label{312.2}
%	\begin{aligned}
%		p&=2, &a\leq9;\\
%		p&=3, &a\leq5;\\
%		p&=5, &a\leq3;\\
%		p&=7,11, &a\leq2;\\
%		p&=13,17,\cdots,83, &a=1.\\
%	\end{aligned}
%\end{equation}
According to (\ref{lx}) and the fact $\zeta_1=3,$ all possible values of $q_0$ and $v$  can be obtained by a little calculation, which listed in Table \ref{Table:312.2}. However, $v$ is not a perfect square for each case, a contradiction.
\begin{table}[h]
	\centering
	\caption{Possible values for $v$  with $q_0=p^a$ \label{Table:312.2}}
	\begin{tabular}{llll}
		\toprule
		$q_0$&$v$&$q_0$&$v$\\
		\midrule
		$2$&$2^3\cdot3\cdot5$&
		$5$&$2^2\cdot3\cdot5^3\cdot7\cdot13$\\
		$8$&$2^9\cdot3^2\cdot5\cdot13\cdot19$&
		$11$&$2^3\cdot3\cdot11^3\cdot37\cdot61$\\
		$17$&$2^2\cdot3^2\cdot5\cdot7\cdot13\cdot17^3\cdot29$&
		$23$&$2^4\cdot3\cdot5\cdot13^2\cdot23^3\cdot53$\\
		$29$&$2^2\cdot3\cdot5\cdot29^3\cdot271\cdot421$&$32$&$2^{15}\cdot3\cdot5^2\cdot11\cdot41\cdot331$\\
		$41$&$2^2\cdot3\cdot7\cdot29^2\cdot41^3\cdot547$&
		$47$&$2^5\cdot3\cdot5\cdot7\cdot13\cdot17\cdot47^3\cdot103$\\
		$53$&$2^2\cdot3^3\cdot5\cdot53^3\cdot281\cdot919$&
		$59$&$2^3\cdot3\cdot5\cdot7\cdot59^3\cdot163\cdot1741$\\
		$71$&$2^4\cdot3^2\cdot71^3\cdot1657\cdot2521$&
		$83$&$2^3\cdot3\cdot5\cdot7\cdot13\cdot53\cdot83^3\cdot2269$\\
		$125$&$2^2\cdot3^2\cdot5^9\cdot7\cdot13\cdot601\cdot5167$&
		$128$&$2^{21}\cdot3\cdot5\cdot29\cdot43\cdot113\cdot5419$\\
		\bottomrule
	\end{tabular}
\end{table}

If $n=4$, then $v=q_0^6(q_0^4+1)(q_0^3+1)(q_0^2+1)/\zeta_2$, where $\zeta_2=(4,q_0+1)$. 
Thus, $\zeta_2=1$, $2$ or $4$.

\medskip
\noindent
{\bf Subcase 1:} If $\zeta_2=1$, then $v=q_0^6(q_0^4+1)(q_0^3+1)(q_0^2+1)$
and $$|\Out(X)||H_0|=2f\cdot q_0^6\cdot(q_0^4-1)(q_0^3-1)(q_0^2-1).$$
Thus we have that $$(v-1,|\Out(X)||H_0|)=(v-1,2f\cdot(q_0^3-1)(q_0^2-1)^2).$$
It follow that $$q_0^6(q_0^4+1)(q_0^3+1)(q_0^2+1)<4f^2\cdot(q_0^3-1)^2(q_0^2-1)^4.$$
Here $f=2a,$ then
$$p^{6a}(p^{4a}+1)(p^{3a}+1)(p^{2a}+1)<16a^2(p^{3a}-1)^2(p^{2a}-1)^4.$$
Hence,	$p=2,a\leq10$. Using the same approach as before, direct computation shows that $v$ is not a perfect square in any case.

\medskip
\noindent
{\bf Subcase 2:} If $\zeta_2=2$, then $v=q_0^6(q_0^4+1)(q_0^3+1)(q_0^2+1)/2$
and $$|\Out(X)||H_0|=4f\cdot q_0^6\cdot(q_0^4-1)(q_0^3-1)(q_0^2-1).$$
In this case $q_0$ is odd. Note that
$$\left(2(v-1),q_0^6\cdot(q_0^4-1)(q_0^3-1)(q_0^2-1)\right)\mid2\cdot\left(2(v-1),(q_0^3-1)(q_0^2-1)^2\right)$$
by the facts $(2(v-1),q_0)=1$ and $(2(v-1),q_0^2+1)=2$.
Thus we have that $$k+1\mid8f(q_0^3-1)(q_0^2-1)^2.$$
It follow that $$q_0^6(q_0^4+1)(q_0^3+1)(q_0^2+1)/2<64f^2\cdot(q_0^3-1)^2(q_0^2-1)^4.$$
Here $f=2a,$ then
$$p^{6a}(p^{4a}+1)(p^{3a}+1)(p^{2a}+1)<128a^2(p^{3a}-1)^2(p^{2a}-1)^4.$$
%According to the inequality, we conclude that 
%\begin{equation}
%	\begin{aligned}
%		p&=3,&a\leq8\\
%		p&=3,&a\leq5\\
%		p&=7,&a\leq3\\
%		p&=11,13,17,19,&a\leq2\\
%		p&=23,29,\cdots,127,&a=1\\
%	\end{aligned}	
%\end{equation}
%By the same method, direct calculation demonstrates that $𝑣$ is never a perfect square in the specified cases.
According to the same calculation as before, we can exclude this subcase.

\medskip
\noindent
{\bf Subcase 3:} If $\zeta_2=4$, then $v=q_0^6(q_0^4+1)(q_0^3+1)(q_0^2+1)/4$
and $$|\Out(X)||H_0|=8f\cdot q_0^6\cdot(q_0^4-1)(q_0^3-1)(q_0^2-1).$$
Since $(4(v-1),q_0)=1$ and $(4(v-1),q_0^2+1)=2$,
thus, $$k+1\mid16f(q_0^3-1)(q_0^2-1)^2.$$

\noindent
It follow that $$q_0^6(q_0^4+1)(q_0^3+1)(q_0^2+1)/4<256f^2\cdot(q_0^3-1)^2(q_0^2-1)^4.$$

\noindent
Note that $f=2a,$ then
$$p^{6a}(p^{4a}+1)(p^{3a}+1)(p^{2a}+1)<1024a^2(p^{3a}-1)^2(p^{2a}-1)^4.$$
%which is true only when
%\begin{equation}
%	\begin{aligned}
%		p&=3,&a\leq10;\\
%		p&=5,&a\leq6;\\
%		p&=7,&a\leq5;\\
%		p&=11,&a\leq4;\\
%		p&=13,17,19,&a\leq3;\\
%		p&=23,29,\cdots61&a\leq2;\\
%		p&=67,71,\cdots,1021,&a=1.\\
%	\end{aligned}	
%\end{equation}
A brief calculation shows that no solutions exist for parameter pairs $(p,a)$ satisfying the condition $v=k^2.$
\end{proof}

\begin{lemma}\rm\label{C6}
	The subgroup $H$ cannot be  $\mathcal{C}_6$-subgroup.
\end{lemma}
\begin{proof}
 Let $H$ be a $\mathcal{C}_6$-subgroup. 
 Here $H$ is of type $\omega^{2i}\cdot Sp(2i,\omega)$, where $n=\omega^i\geq3$ for some prime $\omega\neq p$ and positive integer $i$, and moreover $f$ is odd and is minimal such that $\omega\cdot(2,\omega)$ divides $q-1=p^f-1$(see \cite[Table 3.5A]{MR1057341}). By \cite[Propositions 4.6.5 and 4.6.6]{MR1057341} and Lemma \ref{bound}, we have 
 $$|H_0|\leq \omega^{2i}\cdot|Sp(2i,\omega)|<\omega^{2i^2+3i}.$$
 Moreover, $\omega<q$ as $\omega\cdot(2,\omega)\mid q-1$. Hence $H_0<q^{2i^2+3i}$, and so by Lemmas \ref{bound} and (\ref{o}) and this fact $|\Out(X)|<2fq$, we get
 \begin{equation}
 	\label{101}
 	 q^{n^2-2}<4f^2\cdot q^{2i^2+9i+4}\cdot\prod \limits_{j=1}^{i}(\omega^{2j} - 1)^2\leq q^{2i^2+9i+6}\cdot\prod \limits_{j=1}^{i}(\omega^{2j} - 1)^2.
 \end{equation}

\noindent
 This together with the fact that $$\prod \limits_{j=1}^{i}(\omega^{2j} - 1)^2<\omega^{2i(i+1)-2}<q^{2i(i+1)-2},$$
 which implies that $\omega^{2i}<4i^2+11i+6.$ A little calculation shows that the last inequality holds only for $(i,\omega)\in\{(1,3),(1,5),(2,2),(3,2)\}.$ 
 
 If $(i,\omega)=(3,2)$, then $n=8$. In this case by (\ref{101}) we have that
 $$q^{13}<4f^2\cdot(2^6-1)^2\cdot(2^4-1)^2\cdot(2^2-1)^2,$$
 which is impossible as in the case $q\geq5$.
 Similarly, $(i,\omega)=(1,5)$ can be ruled out.
 %If $(i,\omega)=(1,5)$, then $n=5$. In this case by (\ref{101}) we have that $$q^{7}<4f^2(5^2-1)^2$$
 %which is impossible as in this case $q\geq11.$
 
 If $(i,\omega)=(1,3)$, then $n=3$. In this case, $H_0\cong3^2:Q_8$ by \cite[p.7]{MR3752029}, and hence $|H_0|= 72$,$$v= \frac{q^3(q^2-1)(q^3-1)}{72\cdot(3,q-1)}.$$ From $k+1\mid432f$ by Lemma \ref{GuanL2}, we get that 
 $$\frac{q^3(q^2-1)(q^3-1)}{216}<k^2=\frac{q^3(q^2-1)(q^3-1)}{72\cdot(3,q-1)}\leq(432f-1)^2<432^2\cdot f^2,$$ it follows that $q\in\{3,5,7,9\}.$ However, in each case, $v$ is not a perfect square.
 
  If $(i,\omega)=(1,4)$, then $n=4$. In this case, $H_0\cong2^4:A_6$ by \cite[p.7]{MR3752029}, and hence $|H_0|=360$, thus, we have that $$v=\frac{q^3(q^4-1)(q^4-q)(q^4-q^2)}{1440\cdot(4,q-1)}.$$
  
  \noindent
  From $k+1\mid2880f$ by Lemma \ref{GuanL2}, we get that 
  $$\frac{q^3(q^4-1)(q^4-q)(q^4-q^2)}{5760}<\frac{q^3(q^4-1)(q^4-q)(q^4-q^2)}{1440\cdot(4,q-1)}<2880^2f^2,$$
  which is impossible as in this case $q\geq17.$
\end{proof}
\begin{lemma}\rm\label{C7}
	The subgroup $H$  cannot be a $\mathcal{C}_7$-subgroup.
\end{lemma}
\begin{proof}
Let $H$ be a $\mathcal{C}_7$-subgroup. Then $H$ is a tensor product subgroup of type $GL(i,q)\wr S_\ell$, where $\ell\geq2,i\geq3$ and $n=i^\ell$ (see \cite[Table 3.5]{MR1057341}). Here by \cite[Proposition 4.7.3]{MR1057341}, we have that $|H_0|\leq|PGL(i,q)|^\ell\cdot\ell!/(n,q-1)$. This together with Lemma \ref{bound} implies that $|H_0|<q^{\ell(i^2-1)}\cdot\ell!.$ So by Lemmas \ref{bound} and (\ref{o}), we have 
  \begin{equation}
  	\label{71}
  	q^{n^2-2}<4f^2\cdot(\ell!)^3\cdot q^{\ell(i^2-1)}\cdot\prod \limits_{j=1}^{i}(q^j - 1)^{2\ell}.
  \end{equation}
  Note that $4f^2\cdot\prod\limits_{j=1}^{i}(q^j-1)^{2\ell}\leq q^{i\ell(i+1)+2}.$ We now fix $\ell$ and let $\Phi_3(i)=i^{2\ell}-\ell\cdot(2i^2+i-1)-4.$ It is straightforward to check that $\Phi_3(i)$ is an increasing function of $i$, for $i\geq3$ and $\ell\geq 2$, and hence $\Phi_3(i)\geq\Phi_3(3)=3^{2\ell}-20\ell-4$. This together with (\ref{71}) implies that $2^{3^{2\ell}-20\ell-4}<(\ell!)^3.$ Hence $3^{2\ell}-20\ell-11<3\ell\log_2(\ell)<3\ell^2$, which is impossible for $\ell\geq2.$
\end{proof}

\begin{lemma}\rm\label{C8}
The subgroup $H$ cannot be a $\mathcal{C}_8$-subgroup.
\end{lemma}
\begin{proof}
 Let now $H$  be a $\mathcal{C}_8$-subgroup. In this case, $H$ is a classical group and by \cite{MR1057341}, we have three possibilities here:

   \medskip
  \noindent
  {\bf Case 1:} $H$ is a symplectic group. Here $n=2i\geq4$ and by \cite[Proposition 4.8.3]{MR1057341}, we have that
  $$|H_0|=|\mkern-1mu \raisebox{0.5ex}{$\hat{\ }$} \mkern-3mu Sp(n,q)|\cdot(n/2,q-1).$$
  
  Now we first consider $n\geq8$. Then by (\ref{e1}) and Lemma \ref{bound}, we have that
  \begin{equation}
  	\label{s1}
  q^{(n^2-n-4)/2}<q^{(n^2-2n)/4}\prod\limits_{j=1}^{(n-2)/2}(q^{n-2j+1}-1)/(n/2,q-1).
   \end{equation}
     There exists a subdegree
     \begin{equation}
     	\label{s2}
     	d=(q^n-1)(q^{n-2}-1)
     \end{equation} 
      by \cite[Table 4]{MR4412358}. This  together the fact that $v-1$ is coprime to $q^{(n/2)-1}-1$ if $n/2$ is even and its coprime to $q^{n/2}-1$ if $n/2$ is odd,  implies that $f(q)=(v-1,d)\leq (q^n-1)(q^{(n/2)-1}+1)$. By Lemma \ref{GuanL1} and (\ref{s1}), we get that
  $$q^{(n^2-n-4)/2}<(q^n-1)^2(q^{(n/2)-1}+1)^2<q^{3n}.$$
  Hence $n^2-7n-4<0$, a contradiction. Thus $n=4$ or $6$.
  
  If $n=4,$ then $d=(q^4-1)(q^2-1)$ by (\ref{s2}), and according to (\ref{e1}), we have that  $$v=\frac{\prod\limits_{j=0}^{3}(q^4-q^j)}{(q-1)(4,q-1)}\cdot\frac{(4,q-1)}{q^4\cdot\prod\limits_{j=1}^{2}(q^{2j}-1)\cdot(2,q-1)}=\frac{(q^2+1)(q^3-1)}{(2,q-1)}.$$
  Assume that $q$ is odd, then 
  $(2(v-1),(q^2+1)(q-1)^2)\mid8$, hence $$(2(v-1),(q^4-1)(q^2-1))\mid 8\cdot(2(v-1),(q+1)^2).$$ It follows that $$k+1\mid8(q+1)^2,$$
which implies 
$\frac{(q^2+1)(q^3-1)}{2}=k^2<64(q+1)^4.$  
 % The inequality holds only for pairs $(p,f)$ as in below:
%\begin{align*}
%	p&=3, &f\leq4;\\
%	p&=5, &f\leq3;\\
%	p&=7,11, &f\leq2;\\
%	p&=13,17,\cdots131, &f=1.\\
%\end{align*}
However, by direct computation, there is no solution, which gives rise to any possible parameters $(v,k)$ for $v=k^2.$
Now suppose that $q$ is even, then $(v-1,(q^2+1)(q-1)^2)=1$, and hence $$(v-1,(q^4-1)(q^2-1))=(v-1,(q+1)^2).$$
By Lemma \ref{GuanC1}, we have that $k+1$ divides $(q+1)^2.$
It follows that $$q^5+q^3-q^2-2<(q+1)^4,$$
which forces $q=2$, but $v$ is not a perfect square in this case.

The case of $n=6$ can likewise be excluded.
%If $n=6,$ then $d=(q^6-1)(q^4-1)$ by (\ref{s2}), and by (\ref{e1}),  $v=q^6(q^5-1)(q^3-1)/\gcd(3,q-1).$ Suppose first that $\gcd(3,q-1)=1$, then $(v-1,(q^3-1)(q-1))=1$, and hence $\gcd(v-1,d)=\gcd(v-1,(q^3+1)(q^2+1)(q+1)).$ Thus we have that $k+1\mid(q^3+1)(q^2+1)(q+1)$ by Lemma \ref{GuanL1}. Then we have that
%\begin{equation}
%	\label{822}
%k^2=q^6(q^5-1)(q^3-1)<(q^3+1)^2(q^2+1)^2(q+1)^2.
%\end{equation}
%Therefore $q=2$, which is a contradiction.
%Now we assume that  $\gcd(3,q-1)=3$, then we have that $(3(v-1),(q^3-1)(q-1))$ divides 9, and hence
%$$(3(v-1),d)\mid9(3(v-1),(q^3+1)(q^2+1)(q+1)).$$
%Thus we conclude that $$k+1\mid9\cdot(q^3+1)(q^2+1)(q+1)$$
%according to Lemma \ref{GuanL1}, which implies that 
%$$q^6(q^5-1)(q^3-1)/3=k^2<81(q^3+1)^2(q^2+1)^2(q+1)^2.$$The inequality holds only for pairs $(p,f)$ as in below:
%\begin{align*}
%	p&=2, &f\leq4;\\
%	p&=3, &f\leq2;\\
%	p&=5,7,11,13, &f=1.\\
%\end{align*}
%However, in each case, we have no possible parameters $(p,f)$ as $v$ is a perfect square.

\medskip
\noindent
{\bf Case 2:} $H$ is an orthogonal group. By \cite[Proposition 4.8.4]{MR1057341}, we have that
\begin{equation}
	\label{o1}
 |H_0|=|PSO^{\epsilon}(n,q)|\cdot(n,2),
 \end{equation}
  where $q$ is odd: this is certainly true if the dimension is odd, and for even dimension it follows from the maximality of $H$ in $G$(see \cite[Table 3.5A]{MR1057341}).

\medskip
\noindent
{\bf Subcase 1:} Assume first $\epsilon=\circ$ and $n=2i+1\geq5$. By Lemmas \ref{bound} and (\ref{o}), we have that 
 
 $$q^{4^{i^2}+4i-1}<4f^2\cdot(2i+1,q-1)^2\cdot q^{2i^2+i}\cdot\prod\limits_{j=1}^{i}(q^{2j}-1)^2.$$ Furthermore, $\prod\limits_{j=1}^{i}(q^{2j}-1)^2<q^{2i^2+2i}$, thus $$q^{i-1}<4f^2\cdot(2i+1,q-1)^2.$$ This inequality holds only for $i\leq4$. 

 If $i=2$, then we conclude that
\begin{equation}
	\label{i2}
q<4\cdot25\cdot f^2
\end{equation}
According to (\ref{e1}), (\ref{o1}) and (\ref{i2}), it is easy to see that $v$ is not a perfect square by simple computations. The cases for $i=3$ or 4 can be ruled out similarly.

% \noindent
%If $i=3,$ then we conclude $q^2<4\cdot49\cdot f^2.$
%This inequality holds only for pairs $(p,f)$ as in below: 
%\begin{equation}
%	\begin{aligned}
%		p&=3, &f\leq3;\\
%		p&=5, &f\leq2;\\
%		p&=7,11,13 &f=1.\\
%	\end{aligned}
%\end{equation}
%These cases can be ruled out by simple computations, as they fail to satisfy $v = k^2.$
 %\noindent
%If $i=4,$ then we conclude $q^3<4\cdot81\cdot f^2.$
%This inequality holds only for pairs $(p,f)$ as in below:
%\begin{equation}
	%\begin{aligned}
	%	p&=3, &f\leq2;\\
	%	p&=5, &f=1;\\
	%\end{aligned}
%\end{equation}
%No suitable $(v,k)$ parameters with $v=k^2$ can be obtained from these given $(p,f)$ values.

If $i=1$, then $H_0=SO_3(q)$ with $q$ odd, it follows that $|H_0|=\frac{q(q^2-1)}{2},$ and hence $$v=\frac{2q^2(q^3-1)}{(3,q-1)}.$$
Here $|\Out(X)||H_0|=fq(q^2-1)(3,q-1).$
Suppose first that $(3,q-1)=1$, then $v=2q^2(q^3-1)$.
Since $$(v-1,fq(q^2-1))=\left(-2q^2+2q-1,fq(q^2-1)\right)$$
and
$$2q(q^2-1)=(2q^2-2q+1)(q+1)-(q+1),$$
we have that 
$$(v-1,fq(q^2-1))\mid f(2q^2-q+1,(q+1)).$$
According to Corollary \ref{GuanC1}, we get that $$k+1\mid f^2(q+1)^2.$$
It follows that
$$k^2=2q^2(q^3-1)\leq (f(q+1)-1)^2<f^2(q+1)^2,$$
Thus $$2p^{2f}(p^{3f}-1)<f^2(p^f+1)^2.$$
However, no odd prime $p$ and positive integer $f$ satisfy this inequality, a contradiction.

\noindent
Now we suppose that $(3,q-1)=3$, then $v=\frac{2}{3}q^2(q^3-1).$ 
Note that  $$(3(v-1),3fq(q^2-1)=3f(-2q^2+2q-3,q(q^2-1))$$
and
$$2q(q^2-1)=(q+1)(2q^2-2q+3)-(3q+3),$$
By Corollary \ref{GuanC1}, we have that
$$k+1\mid9f(q+1),$$
and hence
$$k^2=\frac{2}{3}q^2(q^3-1)\leq(9f(q+1)-1)^2<81f^2(q+1)^2.$$
Thus, we conclude that $q\in\{3,5\}$, $v=2^2\cdot13$ or $2^3\cdot5^2\cdot31/3$, it is easy to see that $v$ is not a perfect square. 
 \noindent
Base on the above analysis, the case $H_0=SO_3$ can be ruled out.

\medskip
\noindent
{\bf Subcase 2:} Assume now that $\epsilon\in\{-,+\}$ and $n=2i\geq6$. By Lemma \ref{bound} and (\ref{o}), we have that 
$$q^{4i^2-2}<4\cdot8\cdot(2i,q-1)^2\cdot q^{2i^2-i}(q^i-\epsilon1)\prod\limits_{j=1}^{i-1}(q^{2j}-1)^2.$$ 
Since $(q^i-\epsilon1)\prod\limits_{j=1}^{i-1}(q^{2j}-1)^2<q^{2i^2-2}$, it follows that $q^{2i-4}<4\cdot8\cdot f^2$ which is true only for $(p,f,i)=(3,1,3),(3,2,3)$ and $(5,1,3)$. However, according to (\ref{e1}) and (\ref{o1}), these values give rise to no possible parameter set as $v$ is not a perfect square.

\medskip
\noindent
{\bf Case 3:} $H$ is a unitary group over the field of $q_0$ elements, where $q=q_0^2.$
Here by \cite[Proposition 4.8.5]{MR1057341}, we have that $|H_0|=|\mkern-1mu \raisebox{0.5ex}{$\hat{\ }$} \mkern-3mu SU(n,q_0)|\cdot(n,q_0-1).$ Then by (\ref{e1}) and Lemma \ref{bound}, we get $v>q_0^{n^2-4}.$ Now we consider the stabilizers in $H$ and $G$ of a non-singular 2-subspace of $V$, there exists a subdegree 
\begin{equation}
	\label{s3}
	d=(q_0^n-(-1)^n)(q_0^{n-1}-(-1)^{n-1})
\end{equation}
 by \cite[Table 4]{MR4412358}. Then we have that
$q_0^{n^2-4}<k^2<d^2$, which implies that $q_0^{n^2-4n-3}<2$. Thus $n^2-4n-3<0$, and hence $n=3$ or $4$. 
 \noindent
If $n=3$, then $$v=q_0^3(q_0^2+1)(q_0^3-1)\cdot\delta,$$
where $\delta=\frac{(3,q_0+1)}{(3,q_0^2-1)}.$ Note that $H_0$  is not a parabolic subgroup, as shown in \cite[Table 8.3]{MR3098485}.
\noindent
{\bf Subcase 1:} Suppose that $\delta=1$, then $v=q_0^3(q_0^2+1)(q_0^3-1)$, and $$|\Out(X)||H_0|=2fq_0^3(q_0^3+1)(q_0^2-1)\cdot\frac{1}{\delta}=2fq_0^3(q_0^3+1)(q_0^2-1).$$
Since  $(v-1,q_0)=1$, 
$$(v-1,|\Out(X)||H_0|)=(q_0^3(q_0^2+1)(q_0^3-1)-1,2f(q_0^3+1)(q_0^2-1)).$$
Since $$q_0^3(q_0^2+1)(q_0^3-1)-1=(q_0^3+1)(q_0^2-1)(q_0^3+2q_0-2)+\phi_1(q_0)$$
and 
$$2(q_0^3+1)(q_0^2-1)=\phi_1(q_0)(q_0+2)+\varphi_1(q_0),$$
where $\phi_1(q_0)=2q_0^4-4q_0^3+2q_0^2+2q_0-3$, $\varphi_1(q_0)=4q_0^3-4q_0^2-q_0+4.$
From $\phi_1(q_0)$ is odd, then $$(v-1,(q_0^2-1)(q_0^3+1))\mid(\phi_1(q_0),\varphi_1(q_0)).$$
By Corollary \ref{GuanC1}, $$k+1\mid2f\varphi_1(q_0).$$
It follows that 
\begin{equation}
	\label{310}
	q_0^5(q_0^3-1)<k^2=q_0^3(q_0^2+1)(q_0^3-1)\leq(2f\varphi_1(q_0)-1)^2<64f^2q_0^6 
\end{equation}
by $4<4q_0^2+q_0$. Let $q_0=p^a$, then $f=2a.$ 
According to (\ref{310}), we have that
\begin{equation}
	\label{st}
	p^{3a}-1<256a^2\cdot p^a.
\end{equation}
 
 %Based on the preceding inequality, we get that
%\begin{equation}
%	\label{20}
%	\begin{aligned}
%		p&=2, &a\leq 6; \\
%		p&=3, &a\leq 3; \\
%		p&=5, &a\leq2;  \\
%		p&=7,11,13, &a=1;\\
%	\end{aligned}
%\end{equation}
\noindent
By little calculation, all possible values of 
$q_0$ and $v$ be obtained which are listed in Table \ref{Table:20}. However, $v$ is not a perfect square.

\begin{table}[h]
	\centering
	\caption{Possible values for $v$ with $q_0 = p^a$\label{Table:20}}
	\begin{tabular}{llll}
		\toprule
		$q_0$   & $v$&$q_0$& $v$ \\
		\midrule
		$2$&$2^3\cdot5\cdot7$&$3$&$2^2\cdot3^3\cdot5\cdot7$\\
		$5$&$2^3\cdot5^3\cdot13\cdot31$&$8$&$2^9\cdot5\cdot7\cdot13\cdot73$\\
		$9$&$2^4\cdot3^6\cdot7\cdot13\cdot41$&$11$&$2^2\cdot5\cdot7\cdot11^3\cdot19\cdot61$\\
		$27$&$2^2\cdot3^9\cdot5\cdot13\cdot73\cdot757$&$32$&$2^{15}\cdot5^2\cdot7\cdot31\cdot41\cdot151$\\
		\bottomrule
	\end{tabular}
\end{table}

\noindent
{\bf Subcase 2:} Suppose that $\delta=\frac{1}{3}$, then $(3,q_0^2-1)=(3,q_0-1)=3$. Thus
$$v=\frac{q_0^3(q_0^2+1)(q_0^3-1)}{3},$$ and
$$|\Out(X)||H_0|=6fq_0^3(q_0^2-1)(q_0^3-1).$$
According to $(3,q_0)=1$ and Lemma \ref{GuanC1}, we have that $$k+1\mid\frac{1}{3}(q_0^3(q_0^2+1)(q_0^3-1)-3,18f(q_0^2-1)(q_0^3-1).$$ 
From
$$q_0^3(q_0^2+1)(q_0^3-1)-3=(q_0^2-1)(q_0^3-1)(q_0^3+2q_0)-(2q_0+3),$$and
$$32(q_0^2-1)(q_0^3-1)=(2q_0+3)(16q_0^4-24q_0^3+20q_0^2-46q_0+69)-175,$$ it follows that $$k+1\mid1050f.$$ Then $$k^2=\frac{q_0^3(q_0^2+1)(q_0^3-1)}{3}\leq(1050f-1)^2<1050^2f^2,$$ and hence 
\begin{equation}
	\label{310_3}
	p^{6a}(p^{4a}+1)(p^{6a}-1)<1050^2\cdot12a^2.
\end{equation}
According to the above inequality (\ref{310_3}), $a=1, q_0=2$, and $v=\frac{2^3\cdot5\cdot7}{3}$, a contradiction. 

If $n=4$, then $v=(q_0^4+1)(q_0^3-1)(q_0^2+1)q_0^6/(4,q_0-1).$ In which case, $d=(q_0^4-1)(q_0^3+1)$ by (\ref{s3}). Thus we have that $$(q_0^4+1)(q_0^3-1)(q_0^2+1)q_0^6/4<v<(q_0^4-1)^2(q_0^3+1)^2$$ 
by Lemma \ref{GuanL1}. Hence $q_0=2$ or $3$, which is impossible as $v$ is not a perfect square.
\end{proof}

\begin{lemma}\rm\label{S}
	The subgroup $H$ cannot be a $\mathcal{S}$-subgroup.
\end{lemma}
\begin{proof}
 Let $H$ be a $\mathcal{S}$-subgroup. By Lemmas \ref{GuanC2} and \ref{bound}, we have that 
$$q^{n^2-2}<|PSL(n,q)|=|X|<|G|<|H|^3.$$
Moreover, by \cite[Theorem 4.1]{MR0779398}, we know that $|H|<q^{3n}.$ Hence $q^{n^2-2}<|H|^3<q^{9n},$ which yields $n^2-2<9n,$ and so $n\leq9$. Further, it follows from  \cite[Corollary 4.1]{MR0779398} that either $n=z(z-1)/2$ for some integer $a$, or $|H|<q^{2n+4}.$ If $n=z(z-1)/2$, then as $n\leq9$, we have $n=3$ or $6.$ If $|H|<q^{2n+4},$ then since $q^{n^2-2}<|H|^3$, we conclude that $q^{n^2-2}<q^{6n+12}$. So $n^2-2<6n+12$ and hence $n\leq7.$
 Therefore, we always have $n\leq7$. Thus $n\in\{3,4,5,6,7\}.$ For this values of $n$, the possibilities for $H_0$ can be read off from \cite[Tables 8.4, 8.9, 8.19, 8.25, 8.36]{MR3098485}. In Table \ref{Table:q}, we list $(X,H_0)$  with the conditions recorded in the fourth column. By Lemmas \ref{bound}
 and (\ref{o}), and this fact that $\Out(X)=2f(n,q-1),$ we have
 \begin{table}[h]
 	\centering
 	\caption{$(X,H_0)$ with conditions for $n\in\{3,4,5,6,7\}$\label{Table:q}}
 	\begin{tabular}{llllccc}
 		\toprule
 		Lines&$X$&$H_0$&Conditions&\\
 		\midrule
 		1&$PSL(3,q)$&$PSL(2,7)$&$q=p$ odd&\\
 		2&$PSL(3,q)$&$A_6$&$q=p$ or $q=p^2$ odd&\\
 		3&$PSL(4,q)$&$PSL(2,7)$&$q=p$ odd&\\
 		4&$PSL(4,q)$&$A_7$&$q=p$&\\
 		5&$PSL(4,q)$&$PSU(4,2)$&$q=p\equiv1$ (mod 6)&\\
 		6&$PSL(5,q)$&$PSL(2,11)$&$q=p$ odd& \\
 		7&$PSL(5,q)$&$M_{11}$&3& \\
 		8&$PSL(5,q)$&$PSU(4,2)$&$q=p\equiv 1$ (mod 6)& \\
 		9&$PSL(6,q)$&$A_6.2_3$&$q=p$ odd& \\
 		10&$PSL(6,q)$&$A_6$&$q=p$ or $p^2$ odd& \\
 		11&$PSL(6,q)$&$PSL(2,11)$&$q=p$ odd& \\
 		12&$PSL(6,q)$&$A_7$&$q=p$ or $p^2$ odd& \\
 		13&$PSL(6,q)$&$PSL(3,4)^{\cdot}2_1^-$&$q=p$ odd& \\
 		14&$PSL(6,q)$&$PSU(3,4)$&$q=p$ odd& \\
 		15&$PSL(6,q)$&$M_{12}$&$q=3$& \\
 		16&$PSL(6,q)$&$PSU(4,3)^{\cdot}2_2^-$&$q=p\equiv 1$ (mod 12)& \\
 		17&$PSL(6,q)$&$PSU(4,3)$&$q=p\equiv 7$ (mod 12)& \\
 		18&$PSL(6,q)$&$PSL(3,q)$&$q$ odd& \\
 		19&$PSL(7,q)$&$PSU(3,3)$&$q=p$ odd& \\
 		\bottomrule
 	\end{tabular}
 \end{table}
 \begin{equation}
 	\label{q}
 q^{n^2-4}<4f^2\cdot|H_0|\cdot|H_0|^2_{p^{\prime}}.
 \end{equation} 
Firstly, we consider $n=5,6$ or $7$. Here we use the result of \cite[Table 5]{MR4908113}, We only need to consider the case where $X=PSL(5,3)$ and $H_0\cong M_{11}$, and this case is not impossible for $v=2^6\cdot3^7\cdot11\cdot19$. Thus $n=5,6$ or 7 can be ruled out.

\noindent
Now we consider $n=3.$

\medskip
\noindent
{\bf Case 1:} $H_0=PSL(2,7)$ with $q$ odd,
then $|H_0|=168$, and 
$$v=\frac{q^3(q^2-1)(q^3-1)}{168\cdot(3,q-1)}.$$
By Lemma \ref{GuanL2}, 
$$k+1\mid1008f.$$
Thus, we get that
$$\frac{q^3(q^2-1)}{504}<k^2=\frac{q^3(q^2-1)(q^3-1)}{168\cdot(3,q-1)}\leq(1008f-1)^2<1008f^2,$$
it follows that
\begin{equation}
	\label{334}
	q\in\{3,5,7,9,11\}.
\end{equation}
%\begin{equation}
%	\label{333}
%	\begin{aligned}
%		p&=3,&f\leq2; \\
%		p&=5,7,11,&f=1.
%	\end{aligned}	
%\end{equation}
\noindent
According to (\ref{334}), a list of possible values for $v$ is provided in the Table \ref{Table:333}. However, $v$ is not a perfect square.
\begin{table}[h]
	\centering
	\caption{Possible values for $v$  with $q=p^f$ \label{Table:333}}
	\begin{tabular}{llcllcl}
		\toprule
		$q$   & $v$ & $v$ perfect square&$q$   & $v$ & $v$ perfect square\\
		\midrule
		$3$&$2\cdot3^2\cdot13/7$&No&$5$&$2^2\cdot5^3\cdot31/7$&No\\
		$7$&$2^2\cdot3\cdot7^2\cdot19$&No&$9$&$2^4\cdot3^4\cdot5\cdot13$&No\\
		$11$&$2\cdot5^2\cdot11^3\cdot19$&No\\
		\bottomrule
	\end{tabular}
\end{table}

\medskip
\noindent
{\bf Case 2:} $H_0=A_6$ with $q$ odd, then $|A_6|=360$, and hence $$v=\frac{q^3(q^2-1)(q^3-1)}{360\cdot(3,q-1)} .$$
According to Lemma \ref{GuanL2}, we have that 
$k+1\mid 2160f.$ Thus $$\frac{q^3(q^2-1)(q^3-1)}{1080}<\frac{q^3(q^2-1)(q^3-1)}{360\cdot(3,q-1)}\leq(2160f-1)^2<2160^2\cdot f^2,$$
it follows that 
\begin{equation}
	\label{336}
	q\in\{3,5,7,11,13\}.
\end{equation}
Therefore, the possible values of $v$ are listed in Table \ref{Table:336}. The case can be ruled out as $v$ is not a perfect square. 
	\begin{table}[h]
	\centering
	\caption{Possible values for $v$  with $q=p^f$\label{Table:336}}
	\begin{tabular}{llcllcl}
		\toprule
		$q$ & $v$&$q$& $v$ \\
		\midrule
		$3$&$2\cdot3\cdot13/5$&$5$&$2^2\cdot5^2\cdot31/3$\\
		$7$&$2^2\cdot7^3\cdot19/5$&$11$&$2^4\cdot3^4\cdot7\cdot13$\\
		$13$&$2\cdot5\cdot7\cdot11^3\cdot19/3$\\
		\bottomrule
	\end{tabular}
\end{table}

\noindent
Similarly, for $n=4$, Lemma \ref{GuanL2} can be applied to rule out this case, just as it was for $n=3.$
%\noindent
%Next we consider $n=4.$

%\medskip
%\noindent
%{\bf Case 1:} $H_0=PSL(2,7)$ with $q$ odd.
%then $|H_0|=168$, by Lemma \ref{GuanL2}, we have $q=2$, a contradiction.

%\medskip
%\noindent
%{\bf Case 2:} $H_0=A_7$ with $q$ odd, then $H_0=2^3\cdot3^2\cdot5\cdot7$, by Lemma \ref{GuanL2}, we have $q=2$ or $3$. Thus $v=8$ or $2^4\cdot3^4\cdot13/7$ respectively, a contradiction. 

%\medskip
%\noindent
%{\bf Case 3:} $H_0=PSL(4,2)$ with $q=p\equiv1$(mod 6), then $H_0=2^6\cdot3^2\cdot5\cdot7$, by Lemma \ref{GuanL2}, we have $q<7$, a contradiction. 
\end{proof}

\medskip
\noindent
{\bf Proof of Theorem \ref{th1.2}} It follows immediately from Lemmas \ref{C1}-\ref{S}. \hfill\qedsymbol 

\section{Reduction and construction for design}
In this section, we will impose the condition $\lambda\mid k$ to conduct a more detailed investigation on block-transitive $t$-$(k^2,k,\lambda)$ designs admitting an automorphism group $G$ with socle being $PSL(n,q)(n\geq3)$. According to the results of Theorem \ref{th1.2}, $\mathcal{D}$ must be a $2$-$(k^2,k,\lambda)$ design. Moreover, if $\lambda=1$,  then $G$ is also flag-transitive(\cite{MR0732265}). However, there exist no flag-transitive $2$-$(k^2,k,1)$ designs(\cite[Lemma 3.5]{MR4516389}). Hence, in the proof of Theorem \ref{th13}, we always assume that $\lambda\geq2.$
The commands mentioned in the proof below are performed by the computer algebra system GAP(\cite{GAP4}) and MAGMA(\cite{MR1484478}).
\begin{lemma}\rm\label{4.1}
		Let $\mathcal{D} = (\mathcal{P},\mathcal{B})$ be a non-trivial $t$-$(k^2,k,\lambda)$ design with $\lambda\mid k$, admitting a block-transitive automorphism $G$ and $X\leq G\unlhd \Aut(X)$, where $X = PSL(n,q)(n\geq3).$ Then $G=PSL(3,3)$ or $PGL(3,3)$, and $\mathcal{D}$ is a $2$-$(12^2,12,\lambda)$ design.
\end{lemma}

\begin{proof}
According to Theorem \ref{th1.2}, $t=2$, $X=PSL(3,3),PSL(4,7)$ or $PSL(5,3)$ and  $k=12,20$ or 11 respectively.
Since $G$ is primitive on $\mathcal{P}$, $G=\Aut(PSL(4,7))$ and $\Aut(PSL(5,3))$ can be ruled out by using MAGMA(\cite{MR1484478}). Thus, we have $PSL(3,3)\unlhd G\leq PGL(3,3)$, $PSL(4,7)\unlhd G\leq PGL(4,7)$, or $G=PSL(5,3).$ By Lemma \ref{21}, we have $b=\frac{\lambda v(v-1)}{k(k-1)}=\lambda k(k+1)$. If $PSL(4,7)\unlhd G \leq PGL(4,7)$ and $k=20$, then $b=20\cdot21\cdot\lambda$ with $\lambda\mid20$ and $\lambda \geq2$. However, $PSL(4,7)\unlhd G \leq PGL(4,7)$ does not have any transitive permutation representations of degree $420\lambda$ by \cite[Tables 8.8-8.9]{MR3098485}. If $G=PSL(5,3)$ and $k=11$, then $b=12\cdot11^2$ as $\lambda \mid 11$ and $\lambda\geq2$.
Similarly, $G\cong PSL(5,3)$ does not have any transitive permutation representations of degree $12\cdot11^2$ by \cite[Tables 8.18-8.19]{MR3098485}. Thus $G=PSL(3,3)$ or $PGL(3,3)$, and $\mathcal{D}$ is a $2$-$(12^2,12,\lambda)$ design.
\end{proof}
 We next will use the software GAP(\cite{GAP4}) and MAGMA(\cite{MR1484478}) to construct all block-transitive $2$-$(144,12,\lambda)$ designs with $\lambda\mid k$, whose automorphism group $G=PSL(3,3)$ or $PGL(3,3)$.
\begin{lemma}\rm\label{4.2}
			Let $\mathcal{D} = (\mathcal{P},\mathcal{B})$ be a non-trivial $2$-$(12^2,12,\lambda)$ design with $\lambda\mid 12$, admitting a block-transitive automorphism $G=PSL(3,3)$. Then $\lambda$ = 3 or 12.
\end{lemma}
\begin{proof}
Suppose that $G=PSL(3,3)$, and let $B\in\mathcal{B}$ be a block of $\mathcal{D}=(\mathcal{P},\mathcal{B})$. Since $G$ is block-transitive, $B$ is a union of $G_B$-orbits on $\mathcal{P}$ and $b=|G:G_B|=|B^G|$. Let $\mathcal{P}=\{1,2,\cdots,144\}$ and $\mathcal{P}(G)=\langle g_1,g_2\rangle$(See ATLAS\cite{MR0827219}) be the image of the permutation representations of $G$ on $\mathcal{P}$, where 
\begin{align*}
	g_1=&(1,2)(3,5)(4,6)(7,11)(8,12)(9,13)(10,14)(15,23)(16,24)(17,25)(18,26)(19,
	27)(20,28)\\
	&(21,29)(22,30)(31,42)(32,43)(33,44)(34,45)(35,46)(36,47)(37,48)
	(38,49)(39,50)(40,51)\\
	&(41,52)(53,68)(54,69)(55,70)(56,71)(57,72)(58,73)
	(59,74)(60,75)(61,76)(62,77)(63,78)\\
	&(64,79)(65,80)(66,81)(67,82)(83,95)
	(84,103)(85,104)(86,91)(87,105)(88,106)(89,107)\\
	&(90,108)(92,93)(94,109)
	(96,110)(97,111)(98,102)(99,112)(100,113)(101,114)(115,117)\\
	&(116,127)(118,
	128)(119,123)(120,129)(121,130)(122,126)(124,131)(125,132)(133,137)\\
	&(134,
	136)(135,139)(138,140)(141,142)(143,144)
\end{align*}
and
\begin{align*}
	g_2=&(2,3,4)(5,7,8)(6,9,10)(11,15,16)(12,17,18)(13,19,20)(14,21,22)(24,31,32)
	(25,33,28)\\
	&(26,34,35)(27,36,37)(29,38,39)(30,40,41)(42,53,50)(43,45,54)
	(44,55,56)(46,57,58)\\
	&(47,59,60)(48,61,62)(49,63,64)(51,65,66)(52,67,68)
	(69,83,84)(70,85,86)(71,87,88)\\
	&(72,89,90)(73,91,92)(75,93,94)(76,95,81)
	(77,96,82)(78,97,98)(79,99,100)(80,101,102)\\
	&(103,115,116)(104,117,114)
	(105,118,108)(106,119,120)(107,121,122)(109,110,112)\\
	&(111,123,124)(113,
	125,126)(128,133,134)(131,135,136)(132,137,138)(139,141,140)\\
	&(142,143,144).
\end{align*}
 Here we have that $\lambda\in\{2,3,4,6,12\}$ as $\lambda\mid12$. Suppose first $\lambda$=2, then $b=312$, and $|G_B|=18.$ The MAGMA(\cite{MR1484478}) command ``Subgroups(G:OrderEqual:=18)" shows that there are five conjugacy classes of subgroups with index 18, denoted respectively by $K_1,K_2,\cdots,K_5$ as representatives. The MAGMA(\cite{MR1484478}) commands ``O:=Orbits(K)" for $K=K_i$($i=1,2,\cdots5$), and ``\#O[j]"($j=1,2,\cdots,10$) show that the 8 orbit lengths of $K_c$($c = 1,2,3,4$) are $18^8$, and the 10 orbit lengths of $K_5$ are $6^3$ and $18^7$, where $x^y$ means that the orbit-length $x$ appears $y$ times. Since $B$ is the union of $G_B$-orbits on $\mathcal{P}$, we only need to consider $K_5$. Using the GAP(\cite{GAP4}) Package ``design", the command ``BlockDesign(v,[B],G))" and ``ALLTDesignLambdas($\mathcal{D}$) = 2", we obtain that  there is no $2$-$(144,12,2)$ design. 

For the cases $\lambda= 3, 4, 6$ and 12, the same method as for $\lambda=2$ was applied. We obtain a unique $2$-$(144,12,3)$ design and 96 different $2$-$(144,12,12)$ designs up to isomorphism. The following are some examples.

\medskip
\noindent
(1) $\mathcal{D}_1=(\mathcal{P},B_1^G)$, where $B_1=\{ 3, 7, 29, 30, 67, 68, 84, 96, 100, 101, 107, 134 \}$, $\lambda=3$.

\smallskip
\noindent
(2) $\mathcal{D}_2=(\mathcal{P},B_2^G)$, where $B_2=\{1, 2, 6, 15, 30, 35, 47, 56, 81, 118, 122, 135 \}$, $\lambda=12$.
\end{proof}
\begin{remark}
	{\rm $\mathcal{D}_1$ is flag-transitive, and the 96 $2$-$(144,12,12)$ designs are not flag-transitive.}
\end{remark}
\begin{lemma}\rm\label{4.3}
	Let $\mathcal{D} = (\mathcal{P},\mathcal{B})$ be a non-trivial $2$-$(12^2,12,\lambda)$ design with $\lambda\mid 12$, admitting a block-transitive automorphism $G=PGL(3,3)$. Then $\lambda=6.$
\end{lemma}
\begin{proof}
	Suppose that $G=PGL(3,3)$, let $\mathcal{P}=\{1,2,\cdots,144\}$ and $\mathcal{P}(G)=\langle h_1,h_2\rangle$(See ATLAS\cite{MR0827219}) be the image of the permutation representations of $G$ on $\mathcal{P}$, where
\begin{align*}
h_1=&(1,2)(3,4)(5,7)(6,8)(10,13)(11,15)(12,17)(16,21)(18,24)(19,26)(20,27)
(22,30)(23,31)\\&(25,34)(28,38)(29,40)(32,43)(33,45)(35,48)(36,49)(39,53)
(41,55)(42,57)(44,60)(46,52)\\
&(47,62)(50,64)(51,66)(54,59)(56,61)(58,74)
(63,75)(65,77)(67,79)(68,76)(69,80)(71,83)\\
&(72,85)(73,86)(78,90)(81,94)
(84,97)(87,101)(88,103)(89,104)(91,106)(92,107)(93,109)\\
&(95,111)(96,112)
(98,115)(99,116)(100,118)(102,120)(108,124)(110,127)(113,130)\\
&(114,132)
(117,134)(119,137)(121,128)(122,123)(125,139)(126,136)(129,141)(131,142)\\
&(133,143)(135,144)
\end{align*}
and
\begin{align*}
	h_2=& (1,3,5,2)(4,6,9,12)(7,10,14,19)(8,11,16,22)(13,18,25,35)(15,20,28,39)
	(17,23,32,44)\\
	&(21,29,34,47)(24,33,46,48)(26,36,50,65)(27,37,51,67)(30,41,
	56,72)(31,42,58,60)\\
	&(38,52,68,74)(40,54,70,82)(43,59,64,62)(45,61,57,73)
	(49,63,76,88)(53,69,81,95)\\
	&(55,71,84,98)(66,78,91,94)(75,87,102,116)(77,
	89,86,100)(79,92,108,125)(80,93,110,128)\\
	&(83,96,113,131)(85,99,117,135)
	(90,105,118,136)(97,114,120,134)(101,119,138,143)\\
	&(103,121,132,141)(104,
	112,106,122)(107,123,115,133)(109,126,130,124)(111,129,137,139)\\
	&(127,140,
	142,144).
\end{align*}
	Here $\lambda\in\{2,3,4,6,12\}.$ Similar as Lemma \ref{4.2}, using MAGMA(\cite{MR1484478}) and GAP(\cite{GAP4}), there exists a unique $2$-$(144,12,6)$ design(up to isomorphism), which is listed below:
	
	\medskip
	\noindent
~~~~~~~	$\mathcal{D}_3=(\mathcal{P},B_3^G)$, where $B_3=\{  30, 31, 40, 44, 56, 67, 71, 84, 85, 93, 122, 125 \}, \lambda=6$.
\end{proof}
\begin{remark}\rm
	In fact, $\mathcal{D}_3$ is also flag-transitive here. The flag-transitive designs obtained in Lemmas \ref{4.2} and \ref{4.3} are consistent with the conclusions of \cite[Example 2.2]{MR4516389}.
\end{remark}

\medskip
\noindent
{\bf Proof of Theorem \ref{th13}} Base on the analysis of Lemmas \ref{4.1}-\ref{4.3}, the Theorem \ref{th13} holds. \qedsymbol

\medskip
\noindent{\bf Data availability} The datasets generated during and/or analyzed during the current study are available from
the corresponding author on reasonable request.

\section*{Declarations}

\noindent{\bf Conflict of interest} The authors declare that they have no known competing financial interests or personal
relationships that could have appeared to influence the work reported in this paper.

\newpage
\end{document}